\documentclass[a4paper]{amsart}
\usepackage{amscd}
\usepackage{amsmath}
\usepackage{amssymb}
\usepackage{amsthm}
\usepackage{mathrsfs}
\usepackage{bbm}
\usepackage{stmaryrd}

\usepackage[T1]{fontenc}

\newif\ifpdf
\ifx\pdfoutput\undefined
   \pdffalse        
\else
   \pdfoutput=1     
   \pdftrue
\fi

\ifpdf
   \usepackage[pdftex]{graphicx}
\else
   \usepackage{graphicx}
\fi

\numberwithin{equation}{section} \swapnumbers

\newtheorem{satz}{Satz}[section]

\newtheorem{theorem}[satz]{Theorem}
\newtheorem{proposition}[satz]{Proposition}
\newtheorem{corollary}[satz]{Corollary}
\newtheorem{lemma}[satz]{Lemma}

\newtheorem{definition}[satz]{Definition}

\newtheorem{remark}[satz]{Remark}

\newtheorem{example}[satz]{Example}

\newcommand{\bbr}{\mathbb{R}}
\newcommand{\bbe}{\mathbb{E}}
\newcommand{\bbn}{\mathbb{N}}
\newcommand{\bbp}{\mathbb{P}}
\newcommand{\bbq}{\mathbb{Q}}
\newcommand{\bbg}{\mathbb{G}}

\newcommand{\bbh}{\mathbb{H}}

\newcommand{\bbz}{\mathbb{Z}}

\newcommand{\bbk}{\mathbb{K}}

\newcommand{\cala}{\mathscr{A}}
\newcommand{\calb}{\mathcal{B}}
\newcommand{\cald}{\mathcal{D}}
\newcommand{\cale}{\mathcal{E}}
\newcommand{\calf}{\mathcal{F}}
\newcommand{\calg}{\mathcal{G}}
\newcommand{\calh}{\mathcal{H}}
\newcommand{\calk}{\mathcal{K}}
\newcommand{\call}{\mathcal{L}}

\newcommand{\calx}{\mathcal{X}}

\newcommand{\Lip}{{\rm Lip}}

\newcommand{\loc}{{\rm loc}}

\newcommand{\supp}{{\rm supp}}
\newcommand{\Id}{{\rm Id}}
\newcommand{\lin}{{\rm lin}}

\newcommand{\ran}{{\rm ran}}
\newcommand{\Int}{{\rm Int}}

\newcommand{\la}{\langle}
\newcommand{\ra}{\rangle}

\newcommand{\bbI}{\mathbbm{1}}

\begin{document}

\hyphenation{rea-li-za-tion pa-ra-me-tri-za-ti-ons pa-ra-me-tri-za-ti-on Schau-der}

\title[Invariant cones for jump-diffusions in infinite dimensions]{Invariant cones for jump-diffusions in infinite dimensions}
\author{Stefan Tappe}
\date{22 November, 2023}
\thanks{I am grateful to Josef Teichmann for fruitful discussions. Moreover, I gratefully acknowledge financial support from the Deutsche Forschungsgemeinschaft (DFG, German Research Foundation) -- project number 444121509.}
\address{Albert Ludwig University of Freiburg, Department of Mathematical Stochastics, Ernst-Zermelo-Stra\ss{}e 1, D-79104 Freiburg, Germany}
\email{stefan.tappe@math.uni-freiburg.de}
\begin{abstract}
In this paper we provide sufficient conditions for stochastic invariance of closed convex cones for stochastic partial differential equations (SPDEs) of jump-diffusion type, and clarify when these conditions are necessary. Our results apply to the positive cone of abstract $L^2$-spaces. Furthermore, we present a series of applications, where we investigate SPDEs arising in natural sciences and economics.
\end{abstract}

\maketitle\thispagestyle{empty}

\section{Introduction}\label{sec-intro}

Consider a semilinear stochastic partial differential equation (SPDE) of jump-diffusion type
\begin{align}\label{SPDE}
\left\{
\begin{array}{rcl}
dr_t & = & ( A r_t + \alpha(r_t) ) dt + \sigma(r_t) dW_t + \int_E \gamma(r_{t-},x) (N(dt,dx) - F(dx)dt) \medskip
\\ r_0 & = & h_0
\end{array}
\right.
\end{align}
driven by a trace class Wiener process $W$ and a Poisson random measure $N$ on some mark space $E$ with compensator $dt \otimes F(dx)$. The state space of the SPDE (\ref{SPDE}) is a separable Hilbert space $H$, and the operator $A$ is the generator of a strongly continuous semigroup $(S_t)_{t \geq 0}$ on $H$.

Let $K \subset H$ be a closed convex cone in $H$. We say that the cone $K$ is invariant for the SPDE (\ref{SPDE}) if for each starting point $h_0 \in K$ the solution process $r$ to (\ref{SPDE}) stays in $K$. Such invariance problems have been studied for diffusions, e.g. in \cite{Kotelenez, Goncharuk, Milian, CD, Filipovic-inv, Jachimiak, Nakayama, Marinelli-1, Marinelli-2}, and for jump-diffusions, e.g. in \cite{P-R-Z, Barski-Zabczyk, Barski, Schmidt-Tappe, FTT-positivity, FTT-manifolds, Tappe-cones}. The present paper may be regarded as a successive paper of \cite{Tappe-cones}.

In order to describe the required background, let us review the main result from \cite{Tappe-cones}. For this purpose, we fix a generating system $G$ of the cone $K$ (see Definition \ref{def-gen-system}), which allows us to express the cone $K$ as
\begin{align*}
K = \bigcap_{h^* \in G} \{ h \in H : \la h^*,h \ra \geq 0 \}.
\end{align*}
We assume that the coefficients $(\alpha,\sigma,\gamma)$ appearing in the SPDE (\ref{SPDE}) are locally Lipschitz and satisfy the linear growth condition, that the semigroup $(S_t)_{t \geq 0}$ is pseudo-contractive (see Definition \ref{def-pseudo-contractive}), and that the cone $K$ is invariant for the semigroup $(S_t)_{t \geq 0}$ (see Definition \ref{def-cone-inv-semi}). We emphasize that the volatility $\sigma$ does not need to be smooth; an assumption which is frequently imposed when dealing with stochastic invariance problems. We define the subset $D \subset G \times K$ as
\begin{align*}
D := \bigg\{ (h^*,h) \in G \times K : \liminf_{t \downarrow 0} \frac{\langle h^*,S_t h \rangle}{t} < \infty \bigg\}.
\end{align*}
Then (see \cite[Thm. 1.1]{Tappe-cones}) the closed convex cone $K$ is invariant for the SPDE (\ref{SPDE}) if and only if we have
\begin{align}\label{main-1-intro}
h + \gamma(h,x) \in K \quad \text{for $F$-almost all $x \in E$,} \quad \text{for all $h \in K$,}
\end{align}
and for all $(h^*,h) \in D$ we have
\begin{align}\label{main-3-intro}
&\liminf_{t \downarrow 0} \frac{\langle h^*,S_t h \rangle}{t} + \langle h^*,\alpha(h) \rangle - \int_E \langle h^*,\gamma(h,x) \rangle F(dx) \geq 0,
\\ \label{main-4-intro} &\langle h^*,\sigma^j(h) \rangle = 0, \quad j \in \bbn.
\end{align}
However, for this result it is assumed that the cone $K$ is generated by an unconditional Schauder basis (see Definition \ref{def-Schauder}); an assumption which is not always satisfied in applications. Apart from that, in applications the drift condition (\ref{main-3-intro}) may not be straightforward to check due to the term with the limes inferior.

The goal of this paper is to deal with these two concerns by relaxing the aforementioned condition on the cone $K$ and simplifying condition (\ref{main-3-intro}).

Concerning the relaxation of the condition that the cone $K$ is generated by an unconditional Schauder basis, in this paper we will merely assume that $K$ is approximately generated by an unconditional Schauder basis (see Definition \ref{def-approx-Schauder}); a concept which is inspired by an approximation procedure implemented in \cite{Milian} (see the proof of Theorem 2 therein). We will show that this condition is satisfied for the cone of nonnegative functions in every $L^2$-space, and, more generally, for the positive cone in every abstract $L^2$-space (see Proposition \ref{prop-abstract-Schauder}).

Regarding the simplification of condition (\ref{main-3-intro}), let us observe that for all $(h^*,h) \in D$ we have $\la h^*,h \ra = 0$ and
\begin{align*}
\liminf_{t \downarrow 0} \frac{\langle h^*,S_t h \rangle}{t} \in [0,\infty].
\end{align*}
Therefore, the simpler condition
\begin{align}\label{main-5-intro}
\langle h^*,\alpha(h) \rangle - \int_E \langle h^*,\gamma(h,x) \rangle F(dx) \geq 0
\end{align}
implies (\ref{main-3-intro}). Thus, let us assume that the jump condition (\ref{main-1-intro}) is satisfied, and that for all $(h^*,h) \in G \times K$ with $\la h^*,h \ra = 0$ we have (\ref{main-5-intro}) and (\ref{main-4-intro}). Then we may expect that the cone $K$ is invariant for the SPDE (\ref{SPDE}), and this is indeed true. We refer to Theorem \ref{thm-general} for the precise statement and more details. This result applies to the situation where $K$ is the positive cone of an abstract $L^2$-space (see Theorem \ref{thm-abstract-L2}).

Of course, conditions (\ref{main-1-intro}), (\ref{main-5-intro}) and (\ref{main-4-intro}) are generally only sufficient, but not necessary for stochastic invariance of $K$, because the limes inferior in (\ref{main-3-intro}) can be strictly positive (see Example \ref{ex-liminf} for such a situation). However, if the semigroup $(S_t)_{t \geq 0}$ is a local semigroup (see Definition \ref{def-local-semigroup}), then the limes inferior vanishes, and conditions (\ref{main-1-intro}), (\ref{main-5-intro}) and (\ref{main-4-intro}) are indeed necessary and sufficient for stochastic invariance of $K$. We refer to Theorem \ref{thm-general-2} for the precise statement and more details. This result applies to the situation where $K$ is the cone of nonnegative functions in an $L^2$-space, and where the adjoint operator $A^*$ is a local operator (see Theorem \ref{thm-G-star}). This is in particular the case if $A = \Delta$ is the Laplace operator; a situation which often occurs in applications. The concept of a local semigroup is also inspired by \cite{Milian} (see Lemma 5 therein), because this condition is fulfilled if $A$ or $A^*$ is a local operator (see Lemma \ref{lemma-local-operators}).

Besides presenting the just outlined main results (Theorem \ref{thm-general} and Theorem \ref{thm-general-2}), we will develop several techniques for transformations of closed convex cones (see Theorems \ref{thm-transformed-cone}, \ref{thm-prod-cone}, \ref{thm-direct-sum} and \ref{thm-inv-matrix}) which will be useful for the treatment of concrete examples. We will use all these findings in order to study a series of applications in Section \ref{sec-applications}, where we investigate several SPDEs arising in natural sciences and economics.

The remainder of this paper is organized as follows. In Section \ref{sec-cones} we review the required prerequisites about closed convex cones, and in Section \ref{sec-self-dual} about self-dual cones. In Section \ref{sec-diffusion} we prove the announced invariance result for diffusion SPDEs, and in Section \ref{sec-jump-diffusion} for the general situation of jump-diffusion SPDEs. In Section \ref{sec-transformations} we present results about isomorphic transformations of closed convex cones, in Section \ref{sec-products} we investigate products of closed convex cones, and in Section \ref{sec-trans-prod} we deal with isomorphic transformations of products of closed convex cones. Afterwards, in Section \ref{sec-L2} we treat $L^2$-spaces as state spaces, and in Section \ref{sec-abstract-L2} we deal with abstract $L^2$-spaces as state spaces. In Section \ref{sec-Hilbert-semigroup} we present examples of Hilbert spaces and semigroups, and in Section \ref{sec-applications} we study a series of applications, where we investigate several SPDEs arising in natural sciences and economics. For convenience of the reader, in Appendix \ref{app-Filipovic-space} we collect the required results about generalized Filipovi\'{c} spaces, and in Appendix \ref{app-proj-L2} we provide the required results about orthogonal projections in $L^2$-spaces.

\section{Closed convex cones in Hilbert spaces}\label{sec-cones}

In this section we review the required prerequisites about closed convex cones. Let $H$ be a Hilbert space. 

\begin{definition}
A subset $K \subset H$ is called a \emph{convex cone} in $H$ if
\begin{align*}
\lambda h + \mu g \in K
\end{align*}
for all $\lambda,\mu \in \bbr_+$ and $h,g \in K$.
\end{definition}

\begin{definition}
For a subset $A \subset H$ we define the subset $A^* \subset H$ as
\begin{align}\label{A-star}
A^* = \bigcap_{h^* \in A} \{ h \in H : \langle h^*,h \rangle \geq 0 \}.
\end{align}
\end{definition}

Note that $A^*$ is always a closed convex cone. For a subset $G \subset H$ we agree on the notations
\begin{align*}
\lin^+ G &:= \bigcup_{n \in \bbn} \bigg\{ \sum_{i=1}^n \lambda_i h_i : \lambda_i \geq 0 \text{ and } h_i \in G \text{ for all $i = 1,\ldots,n$} \bigg\},
\\ \overline{\lin}^+ G &:= \overline{\lin^+ G}.
\end{align*}

\begin{lemma}\label{lemma-F-G-star}
Let $G \subset H$ be a subset and set $F := \overline{\lin}^+ G$. Then we have $G^* = F^*$.
\end{lemma}

\begin{proof}
The inclusion $F^* \subset G^*$ is obvious. For the proof of the converse inclusion, let $h \in G^*$ be arbitrary. Then we have $\la h^*,h \ra \geq 0$ for all $h^* \in G$. This implies $\la h^*,h \ra \geq 0$ for all $h^* \in F$, and hence $h \in F^*$.
\end{proof}

Let $K \subset H$ be a closed convex cone. Then its \emph{dual cone} $K^*$ is also a closed convex cone in $H$, and we have $K = K^{**}$. The cone $K$ is called \emph{self-dual} if $K = K^*$.

\begin{remark}
The simplest examples of closed convex cones are $H$ and $\{ 0 \}$. Note that their dual cones are given by $H^* = \{ 0 \}$ and $\{ 0 \}^* = H$.
\end{remark}

\begin{definition}\label{def-gen-system}
A subset $G \subset K^*$ is called a \emph{generating system} of the cone $K$ if we have $G^* = K$; that is
\begin{align*}
K = \bigcap_{h^* \in G} \{ h \in H : \langle h^*,h \rangle \geq 0 \}.
\end{align*}
In this case, we also use the notation $(K,G)$ for the closed convex cone $K$, indicating that $G$ is a generating system of $K$.
\end{definition}

\begin{remark}
In \cite{Tappe-cones} we have denoted a generating system by $G^*$. In this paper, we choose the notation $G$ in order to avoid confusions with (\ref{A-star}).
\end{remark}

\begin{lemma}\label{lemma-lin-generating-system}
Let $G \subset K^*$ be a subset. Then $G$ is a generating system of $K$ if and only if $F := \overline{\lin}^+ G$ is a generating system of $K$.
\end{lemma}

\begin{proof}
This is a consequence of Lemma \ref{lemma-F-G-star}.
\end{proof}

For what follows, let $(K,G)$ be a closed convex cone. We call a system $\{ e_k^*,e_k \}_{k \in \bbn}$ an unconditional Schauder basis of $H$ if for each $h \in H$ we have the representation
\begin{align}\label{series}
h = \sum_{k \in \bbn} \la e_k^*,h \ra e_k,
\end{align}
the series (\ref{series}) converges unconditionally, and for every sequence $(h_k)_{k \in \bbn} \subset \bbr$ such that
\begin{align}\label{series-2}
h = \sum_{k \in \bbn} h_k e_k,
\end{align}
and the series (\ref{series-2}) converges unconditionally, we have $h_k = \la e_k^*,h \ra$ for all $k \in \bbn$.

\begin{definition}\label{def-Schauder}
We say that a closed convex cone $(K,G)$ is \emph{generated by an unconditional Schauder basis} if there is an unconditional Schauder basis $\{ e_k^*, e_k \}_{k \in \bbn}$ of $H$ such that
\begin{align*}
G \subset \bigcup_{k \in \bbn} \lin \{ e_k^* \}.
\end{align*}
\end{definition}

\begin{remark}
Recalling Lemma \ref{lemma-lin-generating-system}, if $(K,G)$ is generated by an unconditional Schauder basis, then there are index sets $I_+,I_- \subset \bbn$ such that
\begin{align*}
G = \{ e_k : k \in I_+ \} \cup \{ -e_k : k \in I_- \}
\end{align*}
is a generating system of $K$. In this case, the closed convex cone $K$ consists of all $h \in H$ of the form (\ref{series-2}) such that $h_k \geq 0$ for all $k \in I_+$ and $h_k \leq 0$ for all $k \in I_-$.
\end{remark}

\begin{definition}\label{def-approx-Schauder}
We say that a closed convex cone $(K,G)$ is \emph{approximately generated by an unconditional Schauder basis} if the following conditions are fulfilled:
\begin{enumerate}
\item There exists a sequence $(K_n,G_n)_{n \in \bbn}$ of closed convex cones which are generated by an unconditional Schauder basis.

\item We have $K \subset K_{n+1} \subset K_n$ and $G_n \subset G_{n+1}$ for each $n \in \bbn$, as well as $K = \bigcap_{n \in \bbn} K_n$.

\item For all $(h^*,h) \in \bigcup_{n \in \bbn} G_n \times K$ with $\la h^*,h \ra = 0$ there exists a sequence $(h_m^*)_{m \in \bbn} \subset G$ such that $h_m^* \to h^*$ as $m \to \infty$ and $\la h_m^*,h \ra = 0$ for each $m \in \bbn$.

\item There are a constant $L > 0$ and a sequence
\begin{align*}
(\pi_n)_{n \in \bbn} \subset \Lip_L(H)
\end{align*}
such that $\pi_n(K_n) \subset K$ for each $n \in \bbn$, and we have $\pi_n \to \Id$.

\item For all $n \in \bbn$ and all $(h^*,h) \in G_n \times K_n$ with $\langle h^*,h \rangle = 0$ we have $\langle h^*,\pi_n(h) \rangle = 0$.
\end{enumerate}
\end{definition}

\begin{remark}
If $\bigcup_{n \in \bbn} G_n \subset G$, then condition (3) from Definition \ref{def-approx-Schauder} is automatically fulfilled.
\end{remark}

\begin{definition}\label{def-inward-pointing}
A function $\alpha : H \to H$ is called \emph{inward pointing at the boundary of $(K,G)$} (in short \emph{inward pointing}) if for all $(h^*,h) \in G \times K$ with $\langle h^*,h \rangle = 0$ we have
\begin{align*}
\langle h^*, \alpha(h) \rangle \geq 0.
\end{align*}
\end{definition}

\begin{definition}\label{def-parallel}
A function $\sigma : H \to H$ is called \emph{parallel at the boundary of $(K,G)$} (in short \emph{parallel}) if for all $(h^*,h) \in G \times K$ with $\langle h^*,h \rangle = 0$ we have
\begin{align*}
\langle h^*, \sigma(h) \rangle = 0.
\end{align*}
\end{definition}

The following auxiliary result shows that the set of all inward pointing functions is a convex cone, and that the set of all parallel functions is a vector space. Its proof is straightforward, and therefore omitted.

\begin{lemma}\label{lemma-sum-inward}
The following statements are true:
\begin{enumerate}
\item For $\lambda,\mu \geq 0$ and two inward pointing functions $\alpha, \beta : H \to H$ the linear combination $\lambda \alpha + \mu \beta$ is also inward pointing.

\item For $\lambda,\mu \in \bbr$ and two parallel functions $\sigma, \tau : H \to H$ the linear combination $\lambda \sigma + \mu \tau$ is also parallel.
\end{enumerate}
\end{lemma}

\begin{lemma}\label{lemma-parallel}
Suppose that the closed convex cone $(K,G)$ is generated by an unconditional Schauder basis. Then the following statements are true:
\begin{enumerate}
\item Let $\alpha : H \to H$ be a function which is inward pointing at the boundary of $(K,G)$. Then for each $n \in \bbn$ the function $\alpha \circ \pi_n : H \to H$ is inward pointing at the boundary of $(K_n,G_n)$.

\item Let $\sigma : H \to H$ be a function which is parallel at the boundary of $(K,G)$. Then for each $n \in \bbn$ the function $\sigma \circ \pi_n : H \to H$ is parallel at the boundary of $(K_n,G_n)$.
\end{enumerate}
\end{lemma}

\begin{proof}
Both statements have a similar proof. Exemplarily, we shall prove the first statement. Let $n \in \bbn$ be arbitrary. Furthermore, let $(h^*,h) \in G_n \times K_n$ with $\langle h^*,h \rangle = 0$ be arbitrary. By Definition \ref{def-approx-Schauder} we have $\langle h^*,\pi_n(h) \rangle = 0$ and $\pi_n(K_n) \subset K$. Note that the latter statement implies $(h^*,\pi_n(h)) \in G_n \times K$. By Definition \ref{def-approx-Schauder} there is a sequence $(h_m^*)_{m \in \bbn} \subset G$ such that $h_m^* \to h^*$ as $m \to \infty$ and $\la h_m^*,\pi_n(h) \ra = 0$ for each $m \in \bbn$. In particular, we have $(h_m^*,\pi_n(h)) \in G \times K$ for each $m \in \bbn$. Since $\alpha$ is inward pointing at the boundary of $(K,G)$, it follows that 
\begin{align*}
\langle h^*,\alpha(\pi_n(h)) \rangle = \lim_{m \to \infty} \langle h_m^*,\alpha(\pi_n(h)) \rangle \geq 0,
\end{align*}
completing the proof.
\end{proof}

Now, let $(S_t)_{t \geq 0}$ be a $C_0$-semigroup on $H$ with generator $A$.

\begin{definition}\label{def-pseudo-contractive}
The semigroup $(S_t)_{t \geq 0}$ is called \emph{pseudo-contractive} (or \emph{quasi-contractive}) if there exists a constant $\beta \geq 0$ such that
\begin{align}\label{growth-semigroup}
\| S_t \| \leq e^{\beta t} \quad \text{for all $t \in \bbr_+$.}
\end{align}
\end{definition}

\begin{definition}\label{def-cone-inv-semi}
We say that the closed convex cone $K$ is invariant for the semigroup $(S_t)_{t \geq 0}$ if $S_t K \subset K$ for all $t \geq 0$.
\end{definition}

For $\lambda > \beta$, where the constant $\beta \geq 0$ stems from the growth estimate (\ref{growth-semigroup}), we define the \emph{resolvent} $R_{\lambda} := (\lambda - A)^{-1}$.

\begin{lemma}\cite[Lemma 2.11]{Tappe-cones}\label{lemma-cone-semi-inv}
The following statements are equivalent:
\begin{enumerate}
\item[(i)] $K$ is invariant for the semigroup $(S_t)_{t \geq 0}$.

\item[(ii)] We have $R_{\lambda} K \subset K$ for all $\lambda > \beta$.
\end{enumerate}
\end{lemma}

For what follows, we define the subset $D \subset G \times K$ as
\begin{align*}
D := \bigg\{ (h^*,h) \in G \times K : \liminf_{t \downarrow 0} \frac{\langle h^*,S_t h \rangle}{t} < \infty \bigg\}.
\end{align*}
Note that for all $(h^*,h) \in D$ we have $\la h^*,h \ra = 0$, but the converse might not be true; see Example \ref{ex-liminf} below.

\begin{definition}\label{def-local-semigroup}
We call the semigroup $(S_t)_{t \geq 0}$ a \emph{local semigroup} relative to $(K,G)$ if for all $(h^*,h) \in G \times K$ with $\la h^*,h \ra = 0$ we have
\begin{align*}
\liminf_{t \downarrow 0} \frac{\la h^*,S_t h \ra}{t} = 0.
\end{align*}
\end{definition}

\begin{remark}\label{rem-local-semigroup}
If $(S_t)_{t \geq 0}$ is a local semigroup relative to $(K,G)$, then we have
\begin{align*}
D = \{ (h^*,h) \in G \times K : \la h^*,h \ra = 0 \}
\end{align*}
as well as
\begin{align*}
\liminf_{t \downarrow 0} \frac{\langle h^*,S_t h \rangle}{t} = 0 \quad \text{for all $(h^*,h) \in D$.}
\end{align*}
\end{remark}

\begin{definition}
We call $A$ a local operator relative to $(K,G)$ if $K \subset \cald(A)$, and for all $(h^*,h) \in G \times K$ with $\la h^*,h \ra = 0$ we have $\la h^*,Ah \ra = 0$.
\end{definition}

\begin{definition}
We call $A^*$ a local operator relative to $(K,G)$ if $G \subset \cald(A^*)$, and for all $(h^*,h) \in G \times K$ with $\la h^*,h \ra = 0$ we have $\la A^* h^*,h \ra = 0$.
\end{definition}

\begin{lemma}\label{lemma-local-operators}
The following statements are true:
\begin{enumerate}
\item If $A$ is a local operator relative to $(K,G)$, then $(S_t)_{t \geq 0}$ is a local semigroup relative to $(K,G)$.

\item If $A^*$ is a local operator relative to $(K,G)$, then $(S_t)_{t \geq 0}$ is a local semigroup relative to $(K,G)$.
\end{enumerate}
\end{lemma}

\begin{proof}
This is an immediate consequence of \cite[Lemma 2.16]{Tappe-cones}.
\end{proof}

\begin{lemma}\label{lemma-generator-transform}
Let $\alpha > 0$ and $\mu \in \bbr$ be arbitrary. We define the operator $B := \alpha A + \mu \Id$ on the domain $\cald(B) = \cald(A)$. Then the following statements are true:
\begin{enumerate}
\item $B$ generates a $C_0$-semigroup $(T_t)_{t \geq 0}$.

\item If $(S_t)_{t \geq 0}$ is pseudo-contractive, then $(T_t)_{t \geq 0}$ is pseudo-contractive as well.

\item If $K$ is invariant for the semigroup $(S_t)_{t \geq 0}$, then it is also invariant for the semigroup $(T_t)_{t \geq 0}$.

\item If $(S_t)_{t \geq 0}$ is a local semigroup relative to $(K,G)$, then $(T_t)_{t \geq 0}$ is also a local semigroup relative to $(K,G)$.
\end{enumerate}
\end{lemma}

\begin{proof}
According to \cite[II.2.2, p. 60]{Engel-Nagel} the operator $B$ generates the $C_0$-semigroup $(T_t)_{t \geq 0}$ given by $T_t = e^{\mu t} S_{\alpha t}$ for each $t \geq 0$. If there is a constant $\beta \geq 0$ such that $\| S_t \| \leq e^{\beta t}$ for each $t \geq 0$, then we obtain
\begin{align*}
\| T_t \| = e^{\mu t} \| S_{\alpha t} \| \leq e^{\mu t} e^{\beta \alpha t} = e^{(\mu + \alpha \beta)t} \quad \text{for each $t \geq 0$.}
\end{align*}
Furthermore, if $S_t K \subset K$ for each $t \geq 0$, then we have
\begin{align*}
T_t K = e^{\mu t} S_{\alpha t} K \subset e^{\mu t} K \subset K \quad \text{for each $t \geq 0$.}
\end{align*}
Moreover, if $(S_t)_{t \geq 0}$ is a local semigroup relative to $(K,G)$, then for all $(h^*,h) \in G \times K$ with $\la h^*,h \ra = 0$ we have
\begin{align*}
\liminf_{t \downarrow 0} \frac{\la h^*,T_t h \ra}{t} = \liminf_{t \downarrow 0} \frac{\la h^*,e^{\mu t} S_{\alpha t} h \ra}{t} = \alpha \liminf_{t \downarrow 0} e^{\mu t} \frac{\la h^*, S_{\alpha t} h \ra}{\alpha t} = 0,
\end{align*}
showing that $(T_t)_{t \geq 0}$ is also a local semigroup relative to $(K,G)$.
\end{proof}

\begin{example}\label{ex-trivial-semigroup}
Suppose that $S_t = \Id$ for each $t \geq 0$; that is $A = 0$. Then $(S_t)_{t \geq 0}$ is a local semigroup relative to $(K,G)$.
\end{example}

\begin{example}
Let $H := \ell^2(\bbn)$ be the Hilbert space consisting of all sequences $h = (h_k)_{k \in \bbn} \subset \bbr$ such that $\sum_{k \in \bbn} |h_k|^2 < \infty$. As in \cite[Example 2.5.4]{Pazy}, let $(S_t)_{t \geq 0}$ be the semigroup given by
\begin{align*}
S_t h := (e^{-kt} h_k)_{k \in \bbn} \quad \text{for $t \geq 0$ and $h = (h_k)_{k \in \bbn} \in H$.}
\end{align*}
We consider the subset
\begin{align*}
K := \{ h = (h_k)_{k \in \bbn} \in H : h_k \geq 0 \text{ for all } k \in \bbn \}
\end{align*}
consisting of all nonnegative sequences. Then (cf. \cite[Sec. 9]{Tappe-cones}) the following statements are true:
\begin{itemize}
\item $(S_t)_{t \geq 0}$ is a $C_0$-semigroup of contractions on $H$. 

\item $K$ is a self-dual closed convex cone in $H$. 

\item $G := \{ e_k : k \in \bbn \}$ is a generating system of $K$. Consequently, then cone $K$ is generated by an unconditional Schauder basis.

\item The cone $K$ is invariant for the semigroup $(S_t)_{t \geq 0}$. 

\item $(S_t)_{t \geq 0}$ is a local semigroup relative to $(K,G)$.
\end{itemize}
\end{example}

\begin{example}\label{ex-liminf}
Let $H := H_w := H_w(\bbr_+)$ be the Filipovi\'{c} space for an admissible weight function $w : \bbr_+ \to [1,\infty)$; see Appendix \ref{app-Filipovic-space}. Moreover,  let $K \subset H$ be the subset of all nonnegative functions. According to Propositions \ref{prop-direct-sum-1} and \ref{prop-Fil-1-dim-semi} the following statements are true:
\begin{itemize}
\item $(S_t)_{t \geq 0}$ is a pseudo-contractive $C_0$-semigroup on $H$.

\item $K$ is a closed convex cone in $H$.

\item The cone $K$ is invariant for the semigroup $(S_t)_{t \geq 0}$.

\item $G = \{ \delta_x : x \in \bbr_+ \}$ is a generating system of $K$.
\end{itemize}
We will show that $(S_t)_{t \geq 0}$ is \emph{not} a local semigroup relative to $(K,G)$. We define the function $h : \bbr_+ \to \bbr_+$ as
\begin{align*}
h(x) := x^{\frac{3}{4}} \bbI_{[0,1]}(x) + \bbI_{(1,\infty)}(x), \quad x \in \bbr_+.
\end{align*}
Then $h$ is absolutely continuous. Furthermore, we have
\begin{align*}
h'(x) = \frac{3}{4} x^{-\frac{1}{4}} \bbI_{[0,1]}(x) \quad \text{for almost every $x \in \bbr_+$,}
\end{align*}
and hence
\begin{align*}
|h'(x)|^2 = \frac{9}{16} x^{-\frac{1}{2}} \bbI_{[0,1]}(x) \quad \text{for almost every $x \in \bbr_+$.}
\end{align*}
Therefore, we have $h \in K$. Now, we set $h^* := \delta_0 \in G$. Then we have $\la h^*,h \ra = h(0) = 0$, but
\begin{align*}
\liminf_{t \downarrow 0} \frac{\la h^*,S_t h \ra}{t} = \liminf_{t \downarrow 0} \frac{h(t)}{t} = \lim_{t \downarrow 0} t^{-\frac{1}{4}} = \infty,
\end{align*}
showing that $(S_t)_{t \geq 0}$ is not a local semigroup relative to $(K,G)$.
\end{example}

\section{Self-dual cones in Hilbert space}\label{sec-self-dual}

In this section we provide the required results about self-dual cones. Let $H$ be a Hilbert space, and let $K \subset H$ be a closed convex cone. Recall that the cone $K$ is called \emph{self-dual} if $K = K^*$; that is
\begin{align*}
K = \bigcap_{h^* \in K} \{ h \in H : \la h^*,h \ra \geq 0 \}.
\end{align*}
For what follows, we assume that the cone $K$ is self-dual.

\begin{lemma}\label{lemma-Schauder-fin-dim}
Let $V \subset H$ be a finite dimensional subspace. Suppose there exists an orthonormal system $G \subset K$ which is a basis of $V$. Then the following statements are true:
\begin{enumerate}
\item We have $V \cap K = \lin^+ G$.

\item The orthogonal projection $\pi_V$ is $K$-invariant.
\end{enumerate}
\end{lemma}

\begin{proof}
Since $G$ is a basis of $V$, we have $\lin^+ G \subset V$. Furthermore, since $G \subset K$ and $K$ is a closed convex cone, we have $\lin^+ G \subset K$. Now, let $h \in V \cap K$ be arbitrary. Since $G \subset K$, we have $\la h,e \ra \geq 0$ for each $e \in G$, and hence
\begin{align*}
h = \sum_{e \in G} \la h,e \ra e \in \lin^+ G,
\end{align*}
proving the first statement. For the proof of the second statement, let $h \in K$ be arbitrary. Since $G \subset K$, we have $\la h,e \ra \geq 0$ for each $e \in G$, and hence
\begin{align*}
\pi_V h = \sum_{e \in G} \la h,e \ra e \in K,
\end{align*}
completing the proof.
\end{proof}

\begin{proposition}\label{prop-self-dual-approx-Schauder}
Let $G \subset K$ be a generating system of $K$ such that the following conditions are fulfilled:
\begin{enumerate}
\item There are an increasing sequence $(V_n)_{n \in \bbn}$ of finite dimensional subspaces and an increasing sequence $(G_n)_{n \in \bbn}$ of orthonormal systems such that for each $n \in \bbn$ the orthonormal system $G_n$ is a basis of $V_n$ with $G_n \subset K$, and we have $\pi_{V_n} \to \Id$ for $n \to \infty$.

\item For all $(h^*,h) \in \bigcup_{n \in \bbn} G_n \times K$ with $\la h^*,h \ra = 0$ there exists a sequence $(h_m^*)_{m \in \bbn} \subset G$ such that $h_m^* \to h^*$ as $m \to \infty$ and $\la h_m^*,h \ra = 0$ for each $m \in \bbn$.
\end{enumerate}
Then $(K,G)$ is approximately generated by an unconditional Schauder basis.
\end{proposition}

\begin{proof}
We define the sequence of closed convex cones $(K_n)_{n \in \bbn}$ as $K_n := (V_n \cap K)^*$ for each $n \in \bbn$. Then we have $K_n^* = V_n \cap K$ and $K \subset K_{n+1} \subset K_n$ for each $n \in \bbn$. Now, let $n \in \bbn$ be arbitrary. According to Lemma \ref{lemma-Schauder-fin-dim} we have
\begin{align*}
K_n^* = V_n \cap K = \lin^+ G_n, 
\end{align*}
and the orthogonal projection $\pi_{V_n}$ is $K$-invariant. Therefore, by Lemma \ref{lemma-lin-generating-system} the set $G_n$ is a generating system of the cone $K_n$. Hence, the sequence $(K_n,G_n)_{n \in \bbn}$ is a sequence of closed convex cones generated by an unconditional Schauder basis. Next, we define the sequence $(\pi_n)_{n \in \bbn} \subset \Lip_1(H)$ as the orthogonal projections
\begin{align*}
\pi_n := \pi_{V_n}, \quad n \in \bbn.
\end{align*}
Then we have $\pi_n \to \Id$. Now, let $n \in \bbn$ be arbitrary. Furthermore, let $h \in K_n$ be arbitrary. We also fix an arbitrary $h^* \in K$. Since $\pi_n$ is $K$-invariant, we have $\pi_n(h^*) \in V_n \cap K = K_n^*$, and hence
\begin{align*}
\la h^*,\pi_n(h) \ra = \la \pi_n(h^*),h \ra \geq 0,
\end{align*}
showing that $\pi_n(h) \in K$. Now, let $n \in \bbn$ and $(h^*,h) \in G_n \times K_n$ with $\la h^*,h \ra = 0$ be arbitrary. Then we have $\pi_n(h^*) = h^*$, and hence
\begin{align*}
\la h^*,\pi_n(h) \ra = \la \pi_n(h^*),h \ra = \la h^*,h \ra = 0.
\end{align*}
Now, we claim that
\begin{align}\label{K-dense}
K = \overline{\lin^+} \bigcup_{n \in \bbn} K_n^*. 
\end{align}
Indeed, let $h \in K$ be arbitrary. We define the sequence $(h_n)_{n \in \bbn} \subset H$ as $h_n := \pi_n h$ for each $n \in \bbn$. Then we have $h_n \in V_n$ for each $n \in \bbn$ and $h_n \to h$ for $n \to \infty$. Since $\pi_n$ is $K$-invariant for each $n \in \bbn$, we even have $h_n \in V_n \cap K = K_n^*$ for each $n \in \bbn$, which proves (\ref{K-dense}). Therefore, by Lemma \ref{lemma-lin-generating-system} we obtain
\begin{align*}
K &= \bigcap_{h \in K} \{ h \in H : \langle h^*,h \rangle \geq 0 \}
\\ &= \bigcap_{n \in \bbn} \bigcap_{h \in K_n^*} \{ h \in H : \langle h^*,h \rangle \geq 0 \} = \bigcap_{n \in \bbn} K_n,
\end{align*}
completing the proof.
\end{proof}

\section{The invariance result for diffusion SPDEs}\label{sec-diffusion}

In this section, we prove the desired invariance result for diffusion SPDEs of the type
\begin{align}\label{SPDE-W}
\left\{
\begin{array}{rcl}
dr_t & = & ( A r_t + \alpha(r_t) ) dt + \sigma(r_t) dW_t \medskip
\\ r_0 & = & h_0.
\end{array}
\right.
\end{align}
From now on, let $(\Omega,\calf,(\calf_t)_{t \in \bbr_+},\bbp)$ be a filtered probability space satisfying the usual conditions. Let $H$ be a separable Hilbert space and let $A : H \supset \cald(A) \to H$ be the infinitesimal generator of a $C_0$-semigroup $(S_t)_{t \geq 0}$ on $H$. Let $U$ be a separable Hilbert space, and let $W$ be an $U$-valued $Q$-Wiener process for some nuclear, self-adjoint, positive definite linear operator $Q \in L_1^{++}(U)$; see \cite[Def. 4.2]{Da_Prato}. There exist an orthonormal basis $\{ e_j \}_{j \in \bbn}$ of $U$ and a sequence $(\lambda_j)_{j \in \bbn} \subset (0,\infty)$ with $\sum_{j \in \bbn} \lambda_j < \infty$ such that
\begin{align*}
Q e_j = \lambda_j e_j \quad \text{for all $j \in \bbn$.}
\end{align*}
The space $U_0 := Q^{1/2}(U)$, equipped with the inner product
\begin{align}\label{inner-prod-U0}
\langle u,v \rangle_{U_0} := \langle Q^{-1/2}u, Q^{-1/2}v \rangle_U,
\end{align}
is another separable Hilbert space. We fix the orthonormal basis $\{ g_j \}_{j \in \bbn}$ of $U_0$ given by $g_j := \sqrt{\lambda_j} e_j$ for each $j \in \bbn$, and denote by $L_2^0(H) := L_2(U_0,H)$ the space of all Hilbert-Schmidt operators from $U_0$ into $H$. Let $\alpha : H \to H$ and $\sigma : H \to L_2^0(H)$ be measurable functions which are locally Lipschitz and satisfy the linear growth condition. This ensures that for each $h_0 \in H$ the diffusion SPDE (\ref{SPDE}) has a unique mild solution; that is, an $H$-valued continuous adapted process $r$, unique up to indistinguishability, such that $\bbp$-almost surely
\begin{align}\label{mild-solution-diffusion}
r_t &= S_t h_0 + \int_0^t S_{t-s} \alpha(r_s) ds + \int_0^t S_{t-s} \sigma(r_s) dW_s, \quad t \in \bbr_+,
\end{align}
see, for example \cite{Da_Prato} or \cite{Atma-book}. According to \cite[Prop. 4.3]{Da_Prato} the sequence $(\beta^j)_{j \in \bbn}$ defined as
\begin{align}\label{beta-j}
\beta^j := \frac{1}{\sqrt{\lambda_j}} \langle W,e_j \rangle_U, \quad j \in \bbn
\end{align}
is a sequence of independent real-valued standard Wiener processes, and according to \cite[Prop. 2.4.5]{Liu-Roeckner} we can write (\ref{mild-solution-diffusion}) equivalently as
\begin{align*}
r_t &= S_t h_0 + \int_0^t S_{t-s} \alpha(r_s) ds + \sum_{j \in \bbn} \int_0^t S_{t-s} \sigma^j(r_s) d\beta_s^j, \quad t \in \bbr_+,
\end{align*}
where for each $j \in \bbn$ the mapping $\sigma^j : H \to H$ is given by $\sigma^j(h) := \sigma(h) g_j$, $h \in H$.

Let $(K,G)$ be a closed convex cone in $H$. The cone $K$ is called \emph{invariant} for the diffusion SPDE (\ref{SPDE-W}) if for each $h_0 \in K$ we have $r \in K$ up to an evanescent set\footnote[1]{A random set $A \subset \Omega \times \mathbb{R}_+$ is called \emph{evanescent} if the set $\{ \omega \in \Omega : (\omega,t) \in A \text{ for some } t \in \mathbb{R}_+ \}$ is a $\mathbb{P}$-nullset, cf. \cite[1.1.10]{Jacod-Shiryaev}.}, where $r$ denotes the mild solution to (\ref{SPDE-W}) with $r_0 = h_0$. We assume that the cone $K$ is invariant for the semigroup $(S_t)_{t \geq 0}$, and that it is approximately generated by an unconditional Schauder basis. For every $\lambda > \beta$ we consider the infinite dimensional SDE
\begin{align}\label{SPDE-W-Y}
\left\{
\begin{array}{rcl}
dr_t^{\lambda} & = & ( A_{\lambda} r_t^{\lambda} + \alpha(r_t^{\lambda}) ) dt + \sigma(r_t^{\lambda}) dW_t \medskip
\\ r_0^{\lambda} & = & h_0,
\end{array}
\right.
\end{align}
where the \emph{Yosida approximation} $A_{\lambda} \in L(H)$ is given by
\begin{align*}
A_{\lambda} := \lambda A R_{\lambda} = \lambda^2 R_{\lambda} - \lambda.
\end{align*}
Furthermore, for all $\lambda > \beta$ and $n \in \bbn$ we consider the infinite dimensional SDE
\begin{align}\label{SPDE-W-n}
\left\{
\begin{array}{rcl}
dr_t^{\lambda,n} & = & ( A_{\lambda} (\pi_n(r_t^{\lambda,n})) + \alpha(\pi_n(r_t^{\lambda,n})) ) dt + \sigma(\pi_n(r_t^{\lambda,n})) dW_t \medskip
\\ r_0^{\lambda,n} & = & h_0,
\end{array}
\right.
\end{align}
where the sequence $(\pi_n)_{n \in \bbn}$ stems from Definition \ref{def-approx-Schauder}.

\begin{lemma}\label{lemma-Yosida-inward}
For each $\lambda > \beta$ the Yosida approximation $A_{\lambda}$ is inward pointing.
\end{lemma}

\begin{proof}
Let $(h^*,h) \in G \times K$ with $\langle h^*,h \rangle = 0$ be arbitrary. By Lemma \ref{lemma-cone-semi-inv} we have $R_{\lambda} K \subset K$, and hence
\begin{align*}
\langle h^*, A_{\lambda} h \rangle = \lambda^2 \langle h^*, R_{\lambda} h \rangle - \lambda \langle h^*,h \rangle \geq 0,
\end{align*}
completing the proof.
\end{proof}

\begin{lemma}\label{lemma-stab-W-1}
Suppose that for each $\lambda > \beta$ the closed convex cone $K$ is invariant for the SDE (\ref{SPDE-W-Y}). Then the closed convex cone $K$ is also invariant for the SPDE (\ref{SPDE-W}).
\end{lemma}

\begin{proof}
Let $h_0 \in K$ be arbitrary. By \cite[Prop. 3.2]{Atma-book}, for every $T \in \bbr_+$ we have
\begin{align*}
\bbe \bigg[ \sup_{t \in [0,T]} \| r_t - r_t^{\lambda} \|^2 \bigg] \to 0 \quad \text{as $\lambda \to \infty$,}
\end{align*}
where $r$ denotes the mild solution to the SPDE (\ref{SPDE-W}) with $r_0 = h_0$, and $r^{\lambda}$ denotes the strong solution to the SDE (\ref{SPDE-W-Y}) with $r_0^{\lambda} = h_0$. Therefore, the same argumentation as in the proof of \cite[Prop. B.3]{Tappe-cones} shows that the closed convex cone $K$ is also invariant for the SPDE (\ref{SPDE-W}).
\end{proof}

\begin{lemma}\label{lemma-stab-W-2}
Let $\lambda > \beta$ be arbitrary. Suppose that for each $n \in \bbn$ the closed convex cone $K_n$ is invariant for the SDE (\ref{SPDE-W-n}). Then the closed convex cone $K$ is invariant for the SDE (\ref{SPDE-W-Y}).
\end{lemma}

\begin{proof}
Let $h_0 \in K$ be arbitrary, and let $N \in \bbn$ be arbitrary. By Definition \ref{def-approx-Schauder} there exists a constant $L > 0$ such that $\pi_n \in \Lip_L(H)$ for all $n \in \bbn$. Hence, by \cite[Prop. B.3]{Tappe-cones} we have
\begin{align*}
\bbe \bigg[ \sup_{t \in [0,T]} \| r_t^{\lambda} - r_t^{\lambda,n} \|^2 \bigg] \to 0 \quad \text{as $n \to \infty$,}
\end{align*}
where $r^{\lambda}$ denotes the strong solution to the SDE (\ref{SPDE-W-Y}) with $r_0^{\lambda} = h_0$, and $r^{\lambda,n}$ denotes the strong solution to the SDE (\ref{SPDE-W-n}) with $r_0^{\lambda,n} = h_0$. Therefore, as in the proof of \cite[Prop. B.3]{Tappe-cones} we show that there exist an event $\bar{\Omega}_N \in \calf$ with $\bbp(\bar{\Omega}_N) = 1$ and a subsequence $(n_k)_{k \in \bbn}$ such that for all $(\omega,t) \in \bar{\Omega}_N \times [0,N]$ we have
\begin{align*}
&r_t^{\lambda,n}(\omega) \in K_n \quad \text{for each $n \in \bbn$} \quad \text{and}
\\ &\lim_{k \to \infty} r_t^{\lambda,n_k}(\omega) = r_t^{\lambda}(\omega).
\end{align*}
By Definition \ref{def-approx-Schauder} we have $K \subset K_{n+1} \subset K_n$ for each $n \in \bbn$, as well as $K = \bigcap_{n \in \bbn} K_n$. Thus, for all $(\omega,t) \in \bar{\Omega}_N \times [0,N]$ we obtain
\begin{align*}
r_t^{\lambda}(\omega) = \lim_{k \to \infty} r_t^{\lambda,n_k}(\omega) \in \bigcap_{k \in \bbn} K_{n_k} = \bigcap_{n \in \bbn} K_n = K.
\end{align*}
Therefore, setting $\bar{\Omega} := \bigcap_{N \in \bbn} \bar{\Omega}_N \in \calf$ we obtain $\bbp(\bar{\Omega}) = 1$ and $r_t^{\lambda}(\omega) \in K$ for all $(\omega,t) \in \bar{\Omega} \times \bbr_+$, showing that $K$ is invariant for the SDE (\ref{SPDE-W-Y}).
\end{proof}

Now, we are ready to prove the desired invariance result for diffusion SPDEs.

\begin{theorem}\label{thm-Wiener}
We suppose that the following conditions are fulfilled:
\begin{enumerate}
\item The coefficients $(\alpha,\sigma)$ are locally Lipschitz and satisfy the linear growth condition.

\item The semigroup $(S_t)_{t \geq 0}$ is pseudo-contractive.

\item The closed convex cone $K$ is invariant for the semigroup $(S_t)_{t \geq 0}$.

\item The closed convex cone $(K,G)$ is approximately generated by an unconditional Schauder basis.
\end{enumerate}
Moreover, we assume that for all $(h^*,h) \in G \times K$ with $\langle h^*,h \rangle = 0$ we have
\begin{align}\label{alpha-inward}
\langle h^*,\alpha(h) \rangle &\geq 0,
\\ \label{sigma-par} \langle h^*,\sigma^j(h) \rangle &= 0, \quad j \in \bbn.
\end{align}
Then the closed convex cone $K$ is invariant for the SPDE (\ref{SPDE-W}).
\end{theorem}

\begin{proof}
By Lemma \ref{lemma-stab-W-1} it suffices to prove that for each $\lambda > \beta$ the closed convex cone $K$ is invariant for the SDE (\ref{SPDE-W-Y}). Thus, let us fix an arbitrary $\lambda > \beta$. By Lemma \ref{lemma-stab-W-2} it suffices to prove that for each $n \in \bbn$ the closed convex cone $K_n$ is invariant for the SDE (\ref{SPDE-W-n}). Thus, let us also fix an arbitrary $n \in \bbn$. Condition (\ref{alpha-inward}) means that $\alpha$ is inward pointing at the boundary of $(K,G)$. Combining Lemmas \ref{lemma-Yosida-inward} and \ref{lemma-sum-inward} we obtain that $A_{\lambda} + \alpha$ is inward pointing at the boundary of $(K,G)$. Therefore, by Lemma \ref{lemma-parallel} the function $A_{\lambda} \circ \pi_n + \alpha \circ \pi_n$ is inward pointing at the boundary of $(K_n,G_n)$. Let $j \in \bbn$ be arbitrary. Condition (\ref{sigma-par}) means that $\sigma^j$ is parallel at the boundary of $(K,G)$. Therefore, by Lemma \ref{lemma-parallel} the function $\sigma^j \circ \pi_n$ is parallel at the boundary of $(K_n,G_n)$. Noting that $K_n$ is generated by an unconditional Schauder basis, applying \cite[Thm. 1.1]{Tappe-cones} (where we notice Example \ref{ex-trivial-semigroup} and Remark \ref{rem-local-semigroup}) we deduce that the closed convex cone $K_n$ is invariant for the SDE (\ref{SPDE-W-n}), completing the proof.
\end{proof}

\section{The invariance result for general jump-diffusion SPDEs}\label{sec-jump-diffusion}

In this section we prove the announced invariance results for general jump-diffusion SPDEs of the type (\ref{SPDE}). In addition to the stochastic framework from Section \ref{sec-diffusion}, let $(E,\cale)$ be a Blackwell space, and let $N$ be a homogeneous Poisson random measure with compensator $dt \otimes F(dx)$ for some $\sigma$-finite measure $F$ on $(E,\cale)$; see \cite[Def. II.1.20]{Jacod-Shiryaev}. We fix a measurable mapping $\gamma : H \times E \to H$. Denoting by $L^2(F) := L^2(E,\cale,F;H)$ the space of all square-integrable functions from $E$ into $H$, we suppose that $\gamma : H \to L^2(F)$ is locally Lipschitz and satisfies the linear growth condition. Together with the assumptions from Section \ref{sec-diffusion}, this ensures that for each $h_0 \in H$ the SPDE (\ref{SPDE}) has a unique mild solution; that is, an $H$-valued c\`{a}dl\`{a}g adapted process $r$, unique up to indistinguishability, such that $\bbp$-almost surely
\begin{align*}
r_t &= S_t h_0 + \int_0^t S_{t-s} \alpha(r_s) ds + \int_0^t S_{t-s} \sigma(r_s) dW_s
\\ &\quad + \int_0^t S_{t-s} \gamma(r_{s-},x) (N(ds,dx) - F(dx)ds), \quad t \in \bbr_+,
\end{align*}
see, for example \cite{MPR} or \cite{FTT-SPDE}.

Let $(K,G)$ be a closed convex cone in $H$. The cone $K$ is called \emph{invariant} for the SPDE (\ref{SPDE}) if for each $h_0 \in K$ we have $r \in K$ up to an evanescent set, where $r$ denotes the mild solution to (\ref{SPDE}) with $r_0 = h_0$.

\begin{theorem}\label{thm-general}
We suppose that the following conditions are fulfilled:
\begin{enumerate}
\item The coefficients $(\alpha,\sigma,\gamma)$ are locally Lipschitz and satisfy the linear growth condition.

\item The semigroup $(S_t)_{t \geq 0}$ is pseudo-contractive.

\item The closed convex cone $K$ is invariant for the semigroup $(S_t)_{t \geq 0}$.

\item The closed convex cone $(K,G)$ is approximately generated by an unconditional Schauder basis.
\end{enumerate}
If we have
\begin{align}\label{main-1}
h + \gamma(h,x) \in K \quad \text{for $F$-almost all $x \in E$,} \quad \text{for all $h \in K$,}
\end{align}
and for all $(h^*,h) \in G \times K$ with $\langle h^*,h \rangle = 0$ we have
\begin{align}\label{main-3}
&\langle h^*,\alpha(h) \rangle - \int_E \langle h^*,\gamma(h,x) \rangle F(dx) \geq 0,
\\ \label{main-4} &\langle h^*,\sigma^j(h) \rangle = 0, \quad j \in \bbn,
\end{align}
then the closed convex cone $K$ is invariant for the SPDE (\ref{SPDE}).
\end{theorem}

\begin{proof}
First, let us assume that $(\alpha,\sigma,\gamma)$ are Lipschitz continuous. Then, by the same arguments as in the proof of \cite[Thm. 7.1]{Tappe-cones} it suffices to prove that for each $B \in \cale$ with $F(B) < \infty$ the closed convex cone $K$ is invariant for the SPDE
\begin{align}\label{SPDE-B-2}
\left\{
\begin{array}{rcl}
dr_t & = & ( A r_t + \alpha_B(r_t) ) dt + \sigma(r_t) dW_t \medskip
\\ r_0 & = & h_0,
\end{array}
\right.
\end{align}
where $\alpha_B : H \to H$ is given by
\begin{align*}
\alpha_B(h) := \alpha(h) - \int_B \gamma(h,x) F(dx), \quad h \in H.
\end{align*}
Let $(h^*,h) \in G \times K$ be arbitrary. By (\ref{main-1}) we have
\begin{align*}
\la h^*,\gamma(h,x) \ra = \la h^*,h \ra + \la h^*,\gamma(h,x) \ra = \la h^*,h+\gamma(h,x) \ra \geq 0.
\end{align*}
Hence, by (\ref{main-3}) we obtain
\begin{align*}
\langle h^*, \alpha_B(h) \rangle = \langle h^*,\alpha(h) \rangle - \int_E \langle h^*,\gamma(h,x) \rangle F(dx) + \int_{E \setminus B} \langle h^*,\gamma(h,x) \rangle F(dx) \geq 0.
\end{align*}
Therefore, by Theorem \ref{thm-Wiener} the closed convex cone $K$ is invariant for the SPDE (\ref{SPDE-B-2}), completing the proof in case of Lipschitz continuous coefficients.

Now, consider the general situation where $(\alpha,\sigma,\gamma)$ are locally Lipschitz and satisfy the linear growth condition. By the same arguments as in the proof of \cite[Thm. 8.1]{Tappe-cones} it suffices that for each $n \in \bbn$ the closed convex cone $K$ is invariant for the SPDE
\begin{align}\label{SPDE-retract-n}
\left\{
\begin{array}{rcl}
dr_t^n & = & ( A r_t^n + \alpha_n(r_t^n) ) dt + \sigma_n(r_t^n) dW_t
\\ && + \int_E \gamma_n(r_{t-}^n,x) (N(dt,dx) - F(dx)dt) \medskip
\\ r_0^n & = & h_0
\end{array}
\right.
\end{align}
where $\alpha_n : H \to H$, $\sigma_n : H \to L_2^0(H)$ and $\gamma_n : H \times E \to H$ are given by
\begin{align*}
\alpha_n := \alpha \circ R_n, \quad \sigma_n := \sigma \circ R_n \quad \text{and} \quad \gamma_n := \gamma \circ R_n.
\end{align*}
Here $R_n : H \to H$ denotes the retraction
\begin{align*}
R_n(h) := \lambda_n(h) h, \quad h \in H,
\end{align*}
where the function $\lambda_n : H \to (0,1]$ is given by
\begin{align*}
\lambda_n(h) := \bbI_{\{ \| h \| \leq n \}} + \frac{n}{\| h \|} \bbI_{\{ \| h \| > n \}}, \quad h \in H.
\end{align*}
Let $h \in K$ be arbitrary. Then we have
\begin{align*}
h + \gamma_n(h,x) = h + \gamma(\lambda_n(h) h,x) = \underbrace{(1 - \lambda_n(h)) h}_{\in K} + \underbrace{\lambda_n(h) h + \gamma(\lambda_n(h) h,x)}_{\in K} \in K
\end{align*}
for $F$-almost all $x \in E$. Now, let $(h^*,h) \in G \times K$ be such that $\la h^*,h \ra = 0$. Then we also have $\la h^*,\lambda_n(h) h \ra = 0$, and hence
\begin{align*}
\langle h^*,\sigma_n^j(h) \rangle = \langle h^*,\sigma^j(\lambda_n(h)h) \rangle = 0, \quad j \in \bbn
\end{align*}
as well as
\begin{align*}
&\langle h^*,\alpha_n(h) \rangle - \int_E \langle h^*,\gamma_n(h,x) \rangle F(dx)
\\ &= \langle h^*,\alpha(\lambda_n(h)h) \rangle - \int_E \langle h^*,\gamma(\lambda_n(h)h,x) \rangle F(dx) \geq 0.
\end{align*}
Consequently, by the first part of this proof, the closed convex cone $K$ is invariant for the SPDE (\ref{SPDE-retract-n}), completing the proof.
\end{proof}

\begin{theorem}\label{thm-general-2}
We suppose that the following conditions are fulfilled:
\begin{enumerate}
\item The coefficients $(\alpha,\sigma,\gamma)$ are locally Lipschitz and satisfy the linear growth condition.

\item The semigroup $(S_t)_{t \geq 0}$ is pseudo-contractive.

\item The closed convex cone $K$ is invariant for the semigroup $(S_t)_{t \geq 0}$.

\item The semigroup $(S_t)_{t \geq 0}$ is a local semigroup relative to $G$.
\end{enumerate}
If the closed convex cone $K$ is invariant for the SPDE (\ref{SPDE}), then we have (\ref{main-1}), and for all $(h^*,h) \in G \times K$ with $\langle h^*,h \rangle = 0$ we have (\ref{main-3}) and (\ref{main-4}).
\end{theorem}

\begin{proof}
Taking into account Remark \ref{rem-local-semigroup}, this is a consequence of \cite[Thm. 3.1]{Tappe-cones}.
\end{proof}

If the drift $\alpha$ of the SPDE (\ref{SPDE}) vanishes, then condition (\ref{main-3}) simplifies as follows.

\begin{proposition}\label{prop-alpha-zero}
Suppose that $\alpha = 0$, and that condition (\ref{main-1}) is fulfilled. Then the following statements are equivalent:
\begin{enumerate}
\item[(i)] For all $(h^*,h) \in G \times K$ with $\langle h^*,h \rangle = 0$ we have (\ref{main-3}).

\item[(ii)] For all $(h^*,h) \in G \times K$ with $\langle h^*,h \rangle = 0$ we have
\begin{align}\label{main-5}
\la h^*, \gamma(h,x) \ra = 0 \quad \text{for $F$-almost all $x \in E$.}
\end{align}
\end{enumerate}
\end{proposition}

\begin{proof}
Let $(h^*,h) \in G \times K$ with $\langle h^*,h \rangle = 0$ be arbitrary. Then for $F$-almost all $x \in E$ we have
\begin{align*}
\la h^*, \gamma(h,x) \ra = \la h^*, h + \gamma(h,x) \ra \geq 0.
\end{align*}
Furthermore, since $\alpha = 0$, condition (\ref{main-3}) can be expressed as
\begin{align*}
\int_E \la h^*,\gamma(h,x) \ra F(dx) \leq 0,
\end{align*}
which provides the stated equivalence (i) $\Leftrightarrow$ (ii).
\end{proof}

\section{Isomorphic transformations of closed convex cones}\label{sec-transformations}

In this section we investigate isomorphic transformations of closed convex cones, and present a related invariance result. Let $H$ and $\bbh$ be separable Hilbert spaces. We start with two auxiliary results about linear isomorphisms.

\begin{lemma}\label{lemma-adj-isom}
Let $T \in L(\bbh,H)$ be an isomorphism. Then the following statements are true:
\begin{enumerate}
\item The adjoint operator $T^* \in L(H,\bbh)$ is also an isomorphism.

\item The operator $(T^{-1})^* \in L(\bbh,H)$ is given by $(T^{-1})^* = (T^*)^{-1}$.
\end{enumerate}
\end{lemma}

\begin{proof}
By \cite[Thm. 4.15]{Rudin} the adjoint operator $T^*$ is one-to-one and $\ran(T^*)$ is closed. Furthermore, by \cite[Thm. 4.12]{Rudin} we have $\ran(T^*)^{\perp} = \{ 0 \}$, showing that $T^*$ is surjective. Moreover, for all $h \in \bbh$ and $g \in H$ we have
\begin{align*}
\la (T^*)^{-1}h,g \ra_H = \la (T^*)^{-1}h,T T^{-1} g \ra_H = \la T^* (T^*)^{-1} h,T^{-1} g \ra_{\bbh} = \la h,T^{-1} g \ra_{\bbh},
\end{align*}
showing that $(T^{-1})^* = (T^*)^{-1}$.
\end{proof}

\begin{lemma}\label{lemma-isom-isom}
Let $T \in L(\bbh,H)$ be an isometric isomorphism. Then the following statements are true:
\begin{enumerate}
\item The adjoint operator $T^* \in L(H,\bbh)$ is also an isometric isomorphism.

\item We have $T^{-1} = T^*$.
\end{enumerate}
\end{lemma}

\begin{proof}
By Lemma \ref{lemma-adj-isom} the adjoint operator $T^* \in L(H,\bbh)$ is an isomorphism. Furthermore, for all $h \in H$ and $g \in \bbh$ we have
\begin{align*}
\la T^{-1} h,g \ra_{\bbh} = \la T T^{-1} h, Tg \ra_H = \la h,Tg \ra_H,
\end{align*}
showing that $T^{-1} = T^*$, and for all $h,g \in H$ we have
\begin{align*}
\la T^* h, T^* g \ra_{\bbh} = \la h, T T^* g \ra_H = \la h,g \ra_H,
\end{align*}
showing that $T^*$ is an isometry.
\end{proof}

Now, we proceed with isomorphic transformations of closed convex cones. For what follows, let $T \in L(\bbh,H)$ be an isomorphism, and set $S := (T^{-1})^* = (T^*)^{-1}$, where the latter identity is a consequence of Lemma \ref{lemma-adj-isom}. Note that $S^* = T^{-1}$. If $T$ is an isometric isomorphism, then by Lemma \ref{lemma-isom-isom} we have $S = T$.

\begin{lemma}\label{lemma-new-cone}
Let $(\bbk,\bbg)$ be a closed convex cone in $\bbh$. Then $(K,G) := (T \bbk, S \bbg)$ is a closed convex cone in $H$.
\end{lemma}

\begin{proof}
It is obvious that $K = T \bbk$ is a closed convex cone in $K$. Furthermore, since $\bbg$ is a generating system of $\bbk$, we obtain
\begin{align*}
K = T \bbk &= T \bigg( \bigcap_{g^* \in \bbg} \{ g \in \bbh : \la g^*,g \ra_{\bbh} \geq 0 \} \bigg) = \bigcap_{g^* \in \bbg} T \big( \{ g \in \bbh : \la g^*,g \ra_{\bbh} \geq 0 \} \big)
\\ &= \bigcap_{g^* \in \bbg} \{ h \in H : \la g^*,T^{-1} h \ra_{\bbh} \geq 0 \} = \bigcap_{g^* \in \bbg} \{ h \in H : \la S g^*, h \ra_H \geq 0 \}
\\ &= \bigcap_{h^* \in G} \{ h \in H : \la h^*,h \ra_H \geq 0 \},
\end{align*}
showing that $G$ is a generating system of $K$.
\end{proof}

\begin{lemma}\label{lemma-transform-self-dual}
Let $\bbk$ be closed convex cone in $\bbh$ which is self-dual, and suppose that $T$ is an isometric isomorphism. Then the closed convex cone $K := T \bbk$ is also self-dual.
\end{lemma}

\begin{proof}
By Lemma \ref{lemma-isom-isom} we have $S = T$. Therefore, by Lemma \ref{lemma-new-cone} the set $K$ is a generating system of the closed convex cone $K$; that is $K^* = K$.
\end{proof}

\begin{lemma}\label{lemma-Schauder-basis}
Let $\{ e_k^*,e_k \}_{k \in \bbn}$ is an unconditional Schauder basis of $\bbh$. Then the system $\{ S e_k^*,T e_k \}_{k \in \bbn}$ is an unconditional Schauder basis of $H$.
\end{lemma}

\begin{proof}
Let $h \in H$ be arbitrary, and set $g := T^{-1} h \in \bbh$. Furthermore, let $(k_i)_{i \in \bbn}$ be an arbitrary numeration of $\bbn$. Then we have
\begin{align*}
g = \sum_{i=1}^{\infty} \la e_{k_i}^*,g \ra_{\bbh} \, e_{k_i},
\end{align*}
and hence
\begin{align*}
h = Tg = \sum_{i=1}^{\infty} \la e_{k_i}^*,g \ra_{\bbh} \, T e_{k_i} = \sum_{i=1}^{\infty} \la e_{k_i}^*,T^{-1} h \ra_{\bbh} \, T e_{k_i} = \sum_{i=1}^{\infty} \la S e_{k_i}^*,h \ra_H \, T e_{k_i},
\end{align*}
proving the representation
\begin{align*}
h = \sum_{k \in \bbn} \la S e_k^*, h \ra_H T e_k.
\end{align*}
In order to prove uniqueness of the series representation, let $(h_k)_{k \in \bbn} \subset \bbr$ be an arbitrary sequence such that
\begin{align*}
\sum_{k=1}^{\infty} h_k T e_k = 0.
\end{align*}
Then we have
\begin{align*}
T \bigg( \sum_{k=1}^{\infty} h_k e_k \bigg) = \sum_{k=1}^{\infty} h_k T e_k = 0.
\end{align*}
Since $T$ is an isomorphism, we obtain
\begin{align*}
\sum_{k=1}^{\infty} h_k e_k = 0.
\end{align*}
Therefore, and since $\{ e_k^*,e_k \}_{k \in \bbn}$ is an unconditional Schauder basis of $H$, we deduce that $h_k = 0$ for all $k \in \bbn$, which finishes the proof.
\end{proof}

\begin{lemma}\label{lemma-Schauder-cone}
Let $(\bbk,\bbg)$ be a closed convex cone in $\bbh$, which is generated by an unconditional Schauder basis. Then the closed convex cone $(K,G) := (T \bbk, S \bbg)$ is also generated by an unconditional Schauder basis.
\end{lemma}

\begin{proof}
There exists an unconditional Schauder basis $\{ e_k^*,e_k \}_{k \in \bbn}$ of $\bbh$ such that
\begin{align*}
\bbg \subset \bigcup_{k \in \bbn} \lin \{ e_k^* \}.
\end{align*}
Therefore, we have
\begin{align*}
G = S \bbg \subset \bigcup_{k \in \bbn} \lin \{ S e_k^* \}.
\end{align*}
Furthermore, by Lemma \ref{lemma-Schauder-basis} the system $\{ S e_k^*,T e_k \}_{k \in \bbn}$ is an unconditional Schauder basis of $H$. This completes the proof.
\end{proof}

\begin{proposition}\label{prop-seq-trans}
Let $(\bbk,\bbg)$ be a closed convex cone in $\bbh$, which is approximately generated by an unconditional Schauder basis. Then the closed convex cone $(K,G) := (T \bbk, S \bbg)$ is also approximately generated by an unconditional Schauder basis.
\end{proposition}

\begin{proof}
Since $(\bbk,\bbg)$ is approximately generated by an unconditional Schauder basis, with analogous notation the conditions from Definition \ref{def-approx-Schauder} are fulfilled. We define $(K_n,G_n) := (T \bbk_n,S \bbg_n)$ for each $n \in \bbn$. By Lemma \ref{lemma-Schauder-cone} the sequence $(K_n,G_n)_{n \in \bbn}$ is a sequence of closed convex cones, which are generated by an unconditional Schauder basis. Furthermore, for each $n \in \bbn$ we have $\bbk \subset \bbk_{n+1} \subset \bbk_n$ and $\bbg_n \subset \bbg_{n+1}$, which implies $K \subset K_{n+1} \subset K_n$ and $G_n \subset G_{n+1}$. Moreover, we have
\begin{align*}
K = T \bbk = T \bigg( \bigcap_{n \in \bbn} \bbk_n \bigg) = \bigcap_{n \in \bbn} T \bbk_n = \bigcap_{n \in \bbn} K_n.
\end{align*}
Let $(h^*,h) \in \bigcup_{n \in \bbn} G_n \times K$ with $\la h^*,h \ra_H = 0$ be arbitrary. We set $(g^*,g) := (S^{-1} h^*, T^{-1} h) \in \bigcup_{n \in \bbn} \bbg_n \times \bbk$. There exists a sequence $(g_m^*)_{m \in \bbn} \subset \bbg$ such that $g_m^* \to g^*$ as $m \to \infty$ and $\la g_m^*,g \ra_{\bbh} = 0$ for each $m \in \bbn$. We define the sequence $(h_m^*)_{m \in \bbn} \subset G$ as $h_m^* := S g_m^*$ for each $m \in \bbn$. Then we have $h_m^* \to h^*$ as $m \to \infty$ and
\begin{align*}
\la h_m^*,h \ra_H = \la S g_m^*, T g \ra_H = \la g_m^*,g \ra_{\bbh} = 0 \quad \text{for each $m \in \bbn$.}
\end{align*}
Next, we define the constant $M := \| T \| L \| T^{-1} \|$ and the sequence
\begin{align*}
(\Pi_n)_{n \in \bbn} \subset \Lip_M(H)
\end{align*}
as $\Pi_n := T \pi_n T^{-1}$ for each $n \in \bbn$. Then we have $\Pi_n \to \Id$ and
\begin{align*}
\Pi_n (K_n) = T \pi_n T^{-1} T \bbk_n = T \pi_n (\bbk_n) \subset T \bbk = K, \quad n \in \bbn. 
\end{align*}
Now, let $n \in \bbn$ and $(h^*,h) \in G_n \times K_n$ with $\la h^*,h \ra_H = 0$ be arbitrary. We set $(g^*,g) := (S^{-1} h^*,T^{-1} h) \in \bbg_n \times \bbk_n$. Then we have
\begin{align*}
\la g^*,g \ra_{\bbh} = \la S^{-1} h^*, T^{-1} h \ra_{\bbh} = \la S^{-1} h^*, S^* h \ra_{\bbh} = \la h^*,h \ra_H = 0.
\end{align*}
Therefore, we obtain
\begin{align*}
\la h^*,\Pi_n(h) \ra_H = \la S g^*, T \pi_n T^{-1} T g \ra_H = \la (T^*)^{-1} g^*, T \pi_n(g) \ra_H = \la g^*,\pi_n(g) \ra_{\bbh} = 0,
\end{align*}
completing the proof.
\end{proof}

Now, we are ready to prove the following invariance result. Consider the stochastic framework from Section \ref{sec-jump-diffusion}, and let $(K,G)$ be a closed convex cone in $H$.

\begin{theorem}\label{thm-transformed-cone}
Suppose that the following conditions are fulfilled:
\begin{enumerate}
\item The coefficients $(\alpha,\sigma,\gamma)$ are locally Lipschitz and satisfy the linear growth condition.

\item The semigroup $(S_t)_{t \geq 0}$ is pseudo-contractive.

\item The closed convex cone $K$ is invariant for the semigroup $(S_t)_{t \geq 0}$.

\item There are another separable Hilbert space $\bbh$, an isomorphism $T \in L(\bbh,H)$, and a closed convex cone $(\bbk,\bbg)$ in $\bbh$, which is approximately generated by an unconditional Schauder basis,  such that $(K,G) = (T \bbk, S \bbg)$, where $S := (T^{-1})^* = (T^*)^{-1}$.
\end{enumerate}
If we have (\ref{main-1}), and for all $(h^*,h) \in G \times K$ with $\langle h^*,h \rangle = 0$ we have (\ref{main-3}) and (\ref{main-4}), then the closed convex cone $K$ is invariant for the SPDE (\ref{SPDE}).
\end{theorem}

\begin{proof}
Taking into account Proposition \ref{prop-seq-trans}, this is a consequence of Theorem \ref{thm-general}.
\end{proof}

\section{Products of closed convex cones}\label{sec-products}

In this section we investigate products of closed convex cones, and present related invariance results. Let $H_1,\ldots,H_m$ be separable Hilbert spaces. We define the new state space $\calh := H_1 \times \ldots \times H_m$, which, equipped with the inner product
\begin{align*}
\la h,g \ra_{\calh} := \sum_{i=1}^m \la h_i,g_i \ra_{H_i}, \quad h,g \in \calh
\end{align*}
is another separable Hilbert space. For each $i=1,\ldots,m$ we introduce the linear operator $\delta_i \in L(H_i,\calh)$ as
\begin{align}\label{delta-i}
(\delta_i h)_j := 
\begin{cases}
h, & \text{if $j=i$,}
\\ 0, & \text{otherwise,}
\end{cases}
\quad \text{for each $j=1,\ldots,m$.}
\end{align}
Note that for each $i=1,\ldots,m$ the adjoint operator $\delta_i^* \in L(\calh,H_i)$ is given by $\delta_i^* h = h_i$. 

\begin{lemma}\label{lemma-gen-product}
For each $i=1,\ldots,m$ let $(K_i,G_i)$ be a closed convex cone in $H_i$. Then
\begin{align}\label{prod-cone-def}
(\calk,\calg) := \bigg( K_1 \times \ldots \times K_m, \bigcup_{i=1}^m \delta_i(G_i) \bigg)
\end{align}
is a closed convex cone in $\calh$.
\end{lemma}

\begin{proof}
It is clear that $\calk$ is a closed convex cone in $\calh$. Furthermore, we have
\begin{align*}
\calk &= \times_{i=1}^d \bigg( \bigcap_{h^* \in G_i} \{ h_i \in H_i : \la h^*,h_i \ra_{H_i} \geq 0 \} \bigg)
\\ &= \bigcap_{i=1}^d \bigcap_{h^* \in G_i} \{ g \in \calh : \la h^*,\delta_i^* g \ra_{H_i} \geq 0 \}
\\ &= \bigcap_{i=1}^d \bigcap_{h^* \in G_i} \{ g \in \calh : \la \delta_i h^*,g \ra_{\calh} \geq 0 \}
\\ &= \bigcap_{g^* \in \calg} \{ g \in \calh : \la g^*,g \ra_{\calh} \geq 0 \},
\end{align*}
showing that $\calg$ is a generating system of $\calk$.
\end{proof}

For the next auxiliary result, we agree on the notation $\bbn_m := \{ 1,\ldots,m \}$.

\begin{lemma}\label{lemma-Schauder-product}
For each $i=1,\ldots,m$ let $\{ e_{ik}^*,e_{ik} \}_{k \in \bbn}$ be an unconditional Schauder basis of $H_i$. Then $\{ \delta_i e_{ik}^*, \delta_i e_{ik} \}_{(i,k) \in \bbn_m \times \bbn}$ is an unconditional Schauder basis of $\calh$. 
\end{lemma}

\begin{proof}
Let $h \in \calh$ be arbitrary. Then for each $i=1,\ldots,m$ there exist unique sequences $(c_{ik})_{k \in \bbn} \subset \bbr$ such that
\begin{align}\label{series-multi}
h_i = \sum_{k \in \bbn} c_{ik} e_{ik},
\end{align}
the series (\ref{series-multi}) converges unconditionally, and we have
\begin{align*}
c_{ik} = \la e_{ik}^*,h_i \ra_{H_i}, \quad k \in \bbn.
\end{align*}
Therefore, we obtain
\begin{align*}
h = \sum_{i=1}^m \delta_i h_i = \sum_{i=1}^m \delta_i \sum_{k=1}^{\infty} c_{ik} e_{ik} = \sum_{k=1}^{\infty} \sum_{i=1}^m c_{ik} \delta_i e_{ik},
\end{align*}
and for all $i=1,\ldots,m$ and $k \in \bbn$ we have
\begin{align*}
c_{ik} = \la e_{ik}^*,h_i \ra_{H_i} = \la e_{ik}^*, \delta_i^* h \ra_{H_i} = \la \delta_i e_{ik}^*, h \ra_{\calh}.
\end{align*}
Therefore $\{ \delta_i e_{ik}^*, \delta_i e_{ik} \}_{(i,k) \in \bbn_m \times \bbn}$ is a Schauder basis of $\calh$, and it remains to show that it is unconditional. According to \cite[Prop. 6.31]{Fabian} there is a constant $K > 0$ such that for all $n \in \bbn$, all $a_1,\ldots,a_n \in \bbr$ and all $\epsilon_1,\ldots,\epsilon_n \in \{ -1,1 \}$ we have
\begin{align*}
\bigg\| \sum_{k=1}^n \epsilon_k a_k e_{ik} \bigg\|_{H_i} \leq K \bigg\| \sum_{k=1}^n a_k e_{ik} \bigg\|_{H_i}, \quad i=1,\ldots,m.
\end{align*}
Therefore, for all $n \in \bbn$, all $a_{ik} \in \bbr$, $i=1,\ldots,m$, $k=1,\ldots,n$ and all $\epsilon_{ik} \in \{ -1,1 \}$, $i=1,\ldots,m$, $k=1,\ldots,n$ we obtain
\begin{align*}
&\bigg\| \sum_{i=1}^m \sum_{k=1}^n \epsilon_{ik} a_{ik} \delta_i e_{ik} \bigg\|_{\calh} = \Bigg( \sum_{i=1}^m \bigg\| \sum_{k=1}^n \epsilon_{ik} a_{ik} e_{ik} \bigg\|_{H_i}^2 \Bigg)^{1/2}
\\ &\leq K \Bigg( \sum_{i=1}^m \bigg\| \sum_{k=1}^n a_{ik} e_{ik} \bigg\|_{H_i}^2 \Bigg)^{1/2} = K \bigg\| \sum_{i=1}^m \sum_{k=1}^n a_{ik} \delta_i e_{ik} \bigg\|_{\calh}.
\end{align*}
Therefore, according to \cite[Prop. 6.31]{Fabian} the system $\{ \delta_i e_{ik}^*, \delta_i e_{ik} \}_{(i,k) \in \bbn_m \times \bbn}$ is an unconditional Schauder basis of $\calh$.
\end{proof}

\begin{lemma}\label{lemma-prod-Schauder-pre}
For each $i=1,\ldots,m$ let $(K_i,G_i)$ be a closed convex cone in $H_i$ which is is generated by an unconditional Schauder basis. Then the closed convex cone $(\calk,\calg)$ given by (\ref{prod-cone-def}) is generated by an unconditional Schauder basis.
\end{lemma}

\begin{proof}
For each $i=1,\ldots,m$ there is an unconditional Schauder basis $\{ e_{ik}^*,e_{ik} \}_{k \in \bbn}$ of $H_i$ such that
\begin{align*}
G_i \subset \bigcup_{k \in \bbn} \lin \{ e_{ik}^* \}.
\end{align*}
Therefore, we have
\begin{align*}
\calg = \bigcup_{i=1}^m \delta_i(G_i) \subset \bigcup_{i=1}^m \bigcup_{k \in \bbn} \lin \{ \delta_i e_{ik}^* \}.
\end{align*}
Furthermore, by Lemma \ref{lemma-Schauder-product} the system $\{ \delta_i e_{ik}^*, \delta_i e_{ik} \}_{(i,k) \in \bbn_m \times \bbn}$ is an unconditional Schauder basis of $\calh$. This completes the proof.
\end{proof}

\begin{proposition}\label{prop-prod-Schauder}
For each $i=1,\ldots,m$ let $(K_i,G_i)$ be a closed convex cone in $H_i$ which is is approximately generated by an unconditional Schauder basis. Then the closed convex cone $(\calk,\calg)$ given by (\ref{prod-cone-def}) is approximately generated by an unconditional Schauder basis.
\end{proposition}

\begin{proof}
Since for each $i=1,\ldots,m$ the cone $(K_i,G_i)$ is approximately generated by an unconditional Schauder basis, with analogous notation the conditions from Definition \ref{def-approx-Schauder} are fulfilled. For each $n \in \bbn$ we define
\begin{align*}
(\calk_n,\calg_n) := \bigg( K_{n1} \times \ldots \times K_{nm}, \bigcup_{i=1}^m \delta_i (G_{ni}) \bigg).
\end{align*}
By Lemma \ref{lemma-prod-Schauder-pre} the sequence $(\calk_n,\calg_n)_{n \in \bbn}$ is a sequence of closed convex cones, which are generated by an unconditional Schauder basis. Furthermore, for each $n \in \bbn$ we have $K_i \subset K_{n+1,i} \subset K_{ni}$, $i=1,\ldots,m$ and $G_{ni} \subset G_{n+1,i}$, $i=1,\ldots,m$, which implies $\calk \subset \calk_{n+1} \subset \calk_n$ and $\calg_n \subset \calg_{n+1}$. Moreover, we have
\begin{align*}
\calk = \times_{i=1}^d K_i = \times_{i=1}^d \bigcap_{n \in \bbn} K_{ni} = \bigcap_{n \in \bbn} \big( \times_{i=1}^d K_{ni} \big) = \bigcap_{n \in \bbn} \calk_n.
\end{align*}
Let $(h^*,h) \in \bigcup_{n \in \bbn} \calg_n \times \calk$ with $\la h^*,h \ra_{\calh}= 0$ be arbitrary. There exist $n \in \bbn$, $i \in \{ 1,\ldots,m \}$ and $g^* \in G_{ni}$ such that $h^* = \delta_i g^*$. Hence, we have
\begin{align*}
\la g^*,h_i \ra_{H_i} = \la g^*, \delta_i^* h \ra_{H_i} = \la \delta_i g^*, h \ra_{\calh} = \la h^*,h \ra_{\calh} = 0.
\end{align*}
There exists a sequence $(h_{im}^*)_{m \in \bbn} \subset G_i$ such that $h_{im}^* \to g^*$ as $m \to \infty$ and $\la h_{im}^*,h_i \ra_{H_i} = 0$ for each $m \in \bbn$. Now, we define the sequence $(h_m^*)_{m \in \bbn} \subset \calg$ as $h_m^* := \delta_i h_{im}^*$ for each $m \in \bbn$. Then we have and $h_m^* = \delta_i h_{im}^* \to \delta_i g^* = h^*$ as $m \to \infty$, and
\begin{align*}
\la h_m^*, h \ra_{\calh} = \la \delta_i h_{im}^*, h \ra_{\calh} = \la h_{im}^*,\delta_i^* h \ra_{H_i} = \la h_{im}^*,h_i \ra_{H_i} = 0 \quad \text{for each $m \in \bbn$.}
\end{align*}
Next, we define the sequence $(\Pi_n)_{n \in \bbn}$ as
\begin{align*}
\Pi_n := ( \pi_{ni} \circ \delta_i^* )_{i=1,\ldots,m}.
\end{align*}
There is a constant $L > 0$ such that $(\pi_{ni})_{n \in \bbn} \in \Lip_L(H_i)$ for each $i=1,\ldots,m$. This implies $(\Pi_n)_{n \in \bbn} \in \Lip_L(\calh)$.
Furthermore, for each $n \in \bbn$ we have
\begin{align*}
\Pi_n(\calk_n) = \pi_{n1}(K_{n1}) \times \ldots \times \pi_{nm}(K_{nm}) \subset K_1 \times \ldots \times K_m = \calk,
\end{align*}
and we have $\Pi_n \to \Id_{\calh}$. Now, let $n \in \bbn$ and $(g^*,g) \in \calg_n \times \calk_n$ be such that $\la g^*,g \ra_{\calh} = 0$. Then there exist $i \in \{ 1,\ldots,m \}$ and $h^* \in G_{ni}$ such that $g^* = \delta_i h^*$. Setting $h := \delta_i^* g = g_i \in K_{ni}$, we have
\begin{align*}
\la h^*,h \ra_{H_i} = \la h^*,\delta_i^* g \ra_{H_i} = \la \delta_i h^*,g \ra_{\calh} = \la g^*,g \ra_{\calh} = 0.
\end{align*}
Therefore, we obtain
\begin{align*}
\la g^*, \Pi_n(g) \ra_{\calh} = \la \delta_i h^*,\Pi_n(g) \ra_{\calh} = \la h^*, \delta_i^* \Pi_n(g) \ra_{H_i} = \la h^*,\pi_{in}(h) \ra_{H_i} = 0,
\end{align*}
completing the proof.
\end{proof}

\begin{lemma}\label{lemma-prod-inner-prod}
For each $i=1,\ldots,m$ let $(K_i,G_i)$ be a closed convex cone in $H_i$, and let $(\calk,\calg)$ be the closed convex cone in $\calh$ given by (\ref{prod-cone-def}). Then for every $h^* \in \calh$ the following statements are equivalent:
\begin{enumerate}
\item[(i)] We have $h^* \in \calg$.

\item[(ii)] There are $i \in \{ 1,\ldots,m \}$ and $g^* \in G_i$ such that $h^* = \delta_i g^*$
\end{enumerate}
If any of the previous two conditions is fulfilled, then we have $\la h^*,h \ra_{\calh} = \la g^*,h_i \ra_{H_i}$ for each $h \in \calk$.
\end{lemma}

\begin{proof}
The equivalence (i) $\Leftrightarrow$ (ii) is an immediate consequence of Lemma \ref{lemma-gen-product}. If $h^* = \delta_i g^*$ for some $i \in \{ 1,\ldots,m \}$ and $g^* \in G_i$, then for all $h \in \calk$ we have
\begin{align*}
\la h^*,h \ra_{\calh} = \la \delta_i g^*,h \ra_{\calh} = \la g^*, \delta_i^* h \ra_{H_i} = \la g^*,h_i \ra_{H_i},
\end{align*}
proving the additional statement.
\end{proof}

Now, we consider the stochastic framework from Section \ref{sec-jump-diffusion} on the state space $\calh = H_1 \times \ldots \times H_m$. For each $i=1,\ldots,m$ let $(K_i,G_i)$ be a closed convex cone in $H_i$. We consider the closed convex cone $\calk = K_1 \times \ldots \times K_m$.

\begin{theorem}\label{thm-prod-cone}
We suppose that the following conditions are fulfilled:
\begin{enumerate}
\item The coefficients $(\alpha,\sigma,\gamma)$ are locally Lipschitz and satisfy the linear growth condition.

\item The semigroup $(S_t)_{t \geq 0}$ is pseudo-contractive.

\item The closed convex cone $\calk$ is invariant for the semigroup $(S_t)_{t \geq 0}$.

\item For each $i=1,\ldots,m$ the closed convex cone $(K_i,G_i)$ is approximately generated by an unconditional Schauder basis.
\end{enumerate}
If we have
\begin{align}\label{main-1-prod}
h + \gamma(h,x) \in \calk \quad \text{for $F$-almost all $x \in E$,} \quad \text{for all $h \in \calk$,}
\end{align}
and for all $i=1,\ldots,m$ and $(h^*,h) \in G_i \times \calk$ with $\langle h^*,h_i \rangle_{H_i} = 0$ we have
\begin{align}\label{main-3-prod}
&\langle h^*,\alpha_i(h) \rangle_{H_i} - \int_E \langle h^*,\gamma_i(h,x) \rangle_{H_i} F(dx) \geq 0,
\\ \label{main-4-prod} &\langle h^*,\sigma_i^j(h) \rangle_{H_i} = 0, \quad j \in \bbn,
\end{align}
then the closed convex cone $\calk$ is invariant for the SPDE (\ref{SPDE}).
\end{theorem}

\begin{proof}
Taking into account Proposition \ref{prop-prod-Schauder} and Lemma \ref{lemma-prod-inner-prod}, the statement is a consequence of Theorem \ref{thm-general}.
\end{proof}

For our next result, let us consider the SPDE (\ref{SPDE}) on the original state space $H$. We assume that the Hilbert space admits a direct sum decomposition $H = H_1 \oplus \ldots \oplus H_m$, where $H_1,\ldots,H_m$ are closed subspaces of $H$ such that $H_i \perp H_j$ for all $i,j = 1,\ldots,m$ with $i \neq j$. For each $i=1,\ldots,m$ let $(K_i,G_i)$ be a closed convex cone in $H_i$. We consider the closed convex cone $K = K_1 \oplus \ldots \oplus K_m$.

\begin{theorem}\label{thm-direct-sum}
We suppose that the following conditions are fulfilled:
\begin{enumerate}
\item The coefficients $(\alpha,\sigma,\gamma)$ are locally Lipschitz and satisfy the linear growth condition.

\item The semigroup $(S_t)_{t \geq 0}$ is pseudo-contractive.

\item The closed convex cone $K$ is invariant for the semigroup $(S_t)_{t \geq 0}$.

\item For each $i=1,\ldots,m$ there exist another separable Hilbert space $\bbh_i$, an isometric isomorphism $T_i : \bbh_i \to H_i$ and a closed convex cone $(\bbk_i,\bbg_i)$ in $\bbh_i$, which is approximately generated by an unconditional Schauder basis, such that $(K_i,G_i) = (T_i \bbk_i,T_i \bbg_i)$.
\end{enumerate}
If we have (\ref{main-1}), and for all $(h^*,h) \in \bigcup_{i=1}^m G_i \times K$ with $\la h^*,h \ra = 0$ we have (\ref{main-3}) and (\ref{main-4}), then the closed convex cone $K$ is invariant for the SPDE (\ref{SPDE}).
\end{theorem}

\begin{proof}
We introduce the product space $\bbh := \bbh_1 \times \ldots \times \bbh_m$. According to Proposition \ref{prop-prod-Schauder} the cone
\begin{align*}
(\bbk,\bbg) := \bigg( \bbk_1 \times \ldots \times \bbk_m, \bigcup_{i=1}^m \delta_i(\bbg_i) \bigg)
\end{align*}
is a closed convex cone in $\bbh$, which is approximately generated by an unconditional Schauder basis. Moreover, the linear mapping $T : \bbh \to H$ given by
\begin{align*}
Th = T_1 h_1 + \ldots + T_m h_m, \quad h \in \bbh
\end{align*}
is an isometric isomorphism, and we have $K = T \bbk$. Hence, by Lemma \ref{lemma-new-cone} the cone $K$ has the generating system
\begin{align*}
G = T \bbg = \bigcup_{i=1}^m T \delta_i \bbg_i = \bigcup_{i=1}^m T_i \bbg_i = \bigcup_{i=1}^m G_i.
\end{align*}
Consequently, the result follows from Theorem \ref{thm-transformed-cone}.
\end{proof}

\section{Isometric transformations of products of closed convex cones}\label{sec-trans-prod}

In this section we consider isometric transformations of products of closed convex cones, and present related invariance results. We consider the state space $\calh := H^m$ for some $m \in \bbn$, where $H$ is a separable Hilbert space. Let $(K_1,G_1),\ldots,(K_m,G_m)$ be closed convex cones in $H$, and let $M \in \bbr^{m \times m}$ be an invertible matrix. Then the subset
\begin{align}\label{cone-matrix}
\calk := \bigcap_{i=1}^m \bigg\{ h \in \calh : \sum_{k=1}^m M_{ik} h_k \in K_i \bigg\}
\end{align}
is a closed convex cone in $\calh$. We recall that for each $i = 1,\ldots,m$ the linear operator $\delta_i \in L(\calh,H)$ is given by (\ref{delta-i}).

\begin{lemma}\label{lemma-inv-matrix-1}
The cone
\begin{align*}
(\bbk,\bbg) := \bigg( K_1 \times \ldots \times K_m, \bigcup_{i=1}^m \delta_i(G_i) \bigg)
\end{align*}
is a closed convex cone in $\calh$, which is approximately generated by an unconditional Schauder basis.
\end{lemma}

\begin{proof}
This is an immediate consequence of Proposition \ref{prop-prod-Schauder}.
\end{proof}

Now, let $R \in L(\calh)$ be the linear isomorphism given by
\begin{align*}
Rh := \bigg( \sum_{k=1}^m M_{ik} h_k \bigg)_{i=1,\ldots,m}, \quad h \in \calh,
\end{align*}
and set $T := R^{-1}$ as well as $S := (T^{-1})^* = (T^*)^{-1}$.

\begin{lemma}\label{lemma-inv-matrix-2}
The set $\calg := \bigcup_{i=1}^m S \delta_i G_i$ is a generating system of the closed convex cone $\calk$, and we have $(\calk,\calg) = (T \bbk,S \bbg)$.
\end{lemma}

\begin{proof}
By Lemmas \ref{lemma-new-cone} and \ref{lemma-inv-matrix-1} the set $S \bbg = \calg$ is a generating system of $\calk$. Moreover, by (\ref{cone-matrix}) we have
\begin{align*}
\calk = \bigcap_{i=1}^m \{ h \in \calh : \delta_i^* Rh \in K_i \} = \{ h \in \calh : Rh \in \bbk \} = T \bbk,
\end{align*}
completing the proof.
\end{proof}

\begin{lemma}\label{lemma-inv-matrix-3}
For all $h^* \in \calh$ the following statements are equivalent:
\begin{enumerate}
\item[(i)] We have $h^* \in \calg$.

\item[(ii)] There are $i \in \{ 1,\ldots,m \}$ and $g^* \in G_i$ such that $h^* = S \delta_i g^*$.
\end{enumerate}
If any of the previous two conditions is fulfilled, then we have 
\begin{align*}
\la h^*,h \ra_{\calh} = \la g^*, \delta_i^* R h \ra_H \quad \text{for each $h \in \calk$.}
\end{align*}
\end{lemma}

\begin{proof}
The equivalence (i) $\Leftrightarrow$ (ii) is an immediate consequence of Lemma \ref{lemma-inv-matrix-2}. Furthermore, note that $S^* = T^{-1} = R$. Hence, if $h^* = S \delta_i g^*$ for some $i \in \{ 1,\ldots,m \}$ and $g^* \in G_i$, then for all $h \in \calk$ we obtain
\begin{align*}
\la h^*,h \ra_{\calh} = \la S \delta_i g^*,h \ra_{\calh} = \la \delta_i g^*,S^* h \ra_{\calh} = \la g^*, \delta_i^* S^* h \ra_{H} = \la g^*, \delta_i^* R h \ra_{H},
\end{align*}
completing the proof.
\end{proof}

Now, we consider the stochastic framework from Section \ref{sec-jump-diffusion} on the state space $\calh = H^m$, and the closed convex cone $\calk$ given by (\ref{cone-matrix}).

\begin{theorem}\label{thm-inv-matrix}
Suppose that the following conditions are fulfilled:
\begin{enumerate}
\item The coefficients $(\alpha,\sigma,\gamma)$ are locally Lipschitz and satisfy the linear growth condition.

\item The semigroup $(S_t)_{t \geq 0}$ is pseudo-contractive.

\item The closed convex cone $\calk$ is invariant for the semigroup $(S_t)_{t \geq 0}$.

\item The closed convex cones $(K_i,G_i)$, $i=1,\ldots,m$ are approximately generated by an unconditional Schauder basis.
\end{enumerate}
If we have (\ref{main-1-prod}), and for all $i=1,\ldots,m$ and $(h^*,h) \in G_i \times \calk$ with
\begin{align*}
\bigg\langle h^*, \sum_{k=1}^m M_{ik} h_k \bigg\rangle = 0
\end{align*}
we have
\begin{align}
&\bigg\langle h^*, \sum_{k=1}^m M_{ik} \alpha_k(h) \bigg\rangle - \int_E \bigg\langle h^*, \sum_{k=1}^m M_{ik} \gamma_k(h,x) \bigg\rangle F(dx) \geq 0,
\\ &\bigg\langle h^*, \sum_{k=1}^m M_{ik} \sigma_k^j(h) \bigg\rangle = 0, \quad j \in \bbn,
\end{align}
then the closed convex cone $\calk$ is invariant for the SPDE (\ref{SPDE}).
\end{theorem}

\begin{proof}
Taking into account Lemmas \ref{lemma-inv-matrix-1}--\ref{lemma-inv-matrix-3}, this is a consequence of Theorem \ref{thm-transformed-cone}.
\end{proof}

Now, we fix a closed convex cone $(K,G)$ in $H$. For two elements $h,g \in H$ we agree to write $g \leq_K h$ if $h-g \in K$. The set
\begin{align*}
\calk := \{ h \in \calh : 0 \leq_K h_1 \leq_K \ldots \leq_K h_m \}
\end{align*}
is a closed convex cone in $\calh$.

\begin{corollary}\label{cor-mon-1}
If we have (\ref{main-1-prod}), for $(h^*,h) \in G \times \calk$ with $\la h^*,h_1 \ra = 0$ we have
\begin{align*}
&\langle h^*,\alpha_1(h) \rangle - \int_E \langle h^*,\gamma_1(h,x)\rangle F(dx) \geq 0,
\\ &\langle h^*,\sigma_1^j(h) \rangle = 0, \quad j \in \bbn,
\end{align*}
and for all $i=2,\ldots,m$ and $(h^*,h) \in G_* \times \calk$ with $\langle h^*, h_i-h_{i-1} \rangle = 0$ we have
\begin{align}\label{cond-mon-1}
&\langle h^*,\alpha_i(h)-\alpha_{i-1}(h) \rangle - \int_E \langle h^*,\gamma_i(h,x) - \gamma_{i-1}(h,x) \rangle F(dx) \geq 0,
\\ \label{cond-mon-2} &\langle h^*,\sigma_i^j(h)-\sigma_{i-1}^j(h) \rangle = 0, \quad j \in \bbn,
\end{align}
then the closed convex cone $\calk$ is invariant for the SPDE (\ref{SPDE}).
\end{corollary}

\begin{proof}
Note that the cone $\calk$ has the representation
\begin{align*}
\calk := \{ h \in \calh : h_1 \in K \} \cap \bigcap_{i=2}^m \{ h \in \calh : h_i - h_{i-1} \in K \}
\end{align*}
Hence, it is of the form (\ref{cone-matrix}) with the matrix $M \in \bbr^{m \times m}$ given by
\begin{align}\label{matrix-mon}
M_{ik} =
\begin{cases}
1, & \text{if $k=i$,}
\\ -1, & \text{if $k=i-1$,}
\\ 0, & \text{otherwise,}
\end{cases}
\end{align}
and $K_i = K$ for all $i=1,\ldots,m$. Therefore, the result is a consequence of Theorem \ref{thm-inv-matrix}.
\end{proof}

Now, we consider the set
\begin{align*}
\calk := \{ h \in \calh : h_1 \leq_K \ldots \leq_K h_m \},
\end{align*}
which is also a closed convex cone in $\calh$.

\begin{corollary}\label{cor-mon-2}
If we have (\ref{main-1-prod}), and for all $i=2,\ldots,m$ and $(h^*,h) \in G \times \calk$ with $\langle h^*, h_i-h_{i-1} \rangle = 0$ we have (\ref{cond-mon-1}) and (\ref{cond-mon-2}), then the closed convex cone $\calk$ is invariant for the SPDE (\ref{SPDE}).
\end{corollary}

\begin{proof}
Note that the cone $\calk$ has the representation
\begin{align*}
\calk := \bigcap_{i=2}^m \{ h \in \calh : h_i - h_{i-1} \in K \}
\end{align*}
Hence, it is of the form (\ref{cone-matrix}) with the matrix $M \in \bbr^{m \times m}$ given by (\ref{matrix-mon}) as well as $K_1 = H$ and $K_i = K$ for all $i=2,\ldots,m$. Therefore, the result is a consequence of Theorem \ref{thm-inv-matrix}.
\end{proof}

\begin{remark}\label{rem-multi-curve}
When modeling multiple yield curves (see, for example \cite{Multi-Curve}), it is desirable to have spreads which are ordered with respect to different tenors. This situation is covered by Corollaries \ref{cor-mon-1} and \ref{cor-mon-2}; see also Remark \ref{rem-multi-curve-2}.
\end{remark}

\section{$L^2$-spaces}\label{sec-L2}

In this section we investigate stochastic invariance of the cone of nonnegative functions in $L^2$-spaces. Let $(X,\calx,\mu)$ be a measure space. We define the Hilbert space $H := L^2(X,\calx,\mu)$ and the closed convex cone $K := H_+$ consisting of all nonnegative functions. Note that the cone $K$ is self-dual. We call the measure space $(X,\calx,\mu)$ \emph{separable} if the Hilbert space $H$ is separable.

\begin{lemma}\label{lemma-metric-separable}
The following statements are true:
\begin{enumerate}
\item If $X$ is a separable metric space equipped with its Borel $\sigma$-algebra $\calx = \calb(X)$, then the measure space $(X,\calx,\mu)$ is separable.

\item In particular, if $X \subset \bbr^d$ is a subset and $\calx = \calb(X)$, then the measure space $(X,\calx,\mu)$ is separable.
\end{enumerate}
\end{lemma}

\begin{proof}
This is a consequence of \cite[Thm. 4.13]{Brezis}.
\end{proof}

A filtration $(\calg_n)_{n \in \bbn}$ on the measure space $(X,\calx,\mu)$ is called \emph{finite} if for each $n \in \bbn$ the $\sigma$-algebra $\calg_n$ only consists of finitely many elements.

\begin{proposition}\label{prop-prob-space-separable}
If $(X,\calx,\mu)$ is a probability space, then the following statements are equivalent:
\begin{enumerate}
\item[(i)] The probability space $(X,\calx,\mu)$ is separable.

\item[(ii)] There exists a finite filtration $(\calg_n)_{n \in \bbn}$ such that for each $h \in H$ we have
\begin{align}\label{conv-proj}
\pi_{L^2(\calg_n)} h \to h \quad \text{in $H$.}
\end{align}
\end{enumerate}
\end{proposition}

\begin{proof}
(i) $\Rightarrow$ (ii): By \cite[Exercise 4.7.63]{Bogachev-vol-1} the Banach space $L^1(X,\calx,\mu)$ is also separable. Hence, by the Proposition in Section IV.6.1 (page 219) from \cite{Malliavin} there exists a finite filtration $(\calg_n)_{n \in \bbn}$ such that for each $h \in H$ we have
\begin{align*}
\bbe[h | \calg_n] \to h \quad \text{in $L^1$.}
\end{align*}
Now, let $h \in H$ be arbitrary. By the martingale convergence theorem (see \cite[Thm. 11.10]{Klenke}) we have
\begin{align*}
\bbe [ h | \calg_n ] \to g \quad \text{in $H$}
\end{align*}
for some $g \in H$. Since convergence in $H = L^2$ implies convergence in $L^1$, we have $g=h$. Furthermore, by Lemma \ref{lemma-MT-prob-space} we have $\bbe [ h | \calg_n ] = \pi_{L^2(\calg_n)} h$ for each $n \in \bbn$, which shows (\ref{conv-proj}).

\noindent(ii) $\Rightarrow$ (i): By assumption there is a countable set $A \subset H$ such that $H = \overline{\lin} A$. Hence $H$ is separable.
\end{proof}

Recall that a set $A \in \calx$ is called a \emph{$\mu$-atom} if $\mu(A) > 0$ and for each $B \in \calx$ with $B \subset A$ we have $\mu(B) = 0$ or $\mu(A \setminus B) = 0$. For a sub-$\sigma$-algebra $\calg \subset \calx$ we denote by $\cala_{\mu}(\calg)$ the set of all $\mu$-atoms from $\calg$.

\begin{proposition}\label{prop-L2-Schauder}
Let $G \subset K$ be a generating system of $K$ such that the following conditions are fulfilled:
\begin{enumerate}
\item There exists a finite filtration $(\calg_n)_{n \in \bbn}$ such that for each $h \in H$ we have
\begin{align*}
\pi_{L^2(\calg_n)} h \to h \quad \text{in $H$.}
\end{align*}

\item For all $n \in \bbn$, all $A \in \cala_{\mu}(\calg_n)$ and all $h \in K$ with $\la \bbI_A,h \ra = 0$ there is a sequence $(h_m^*)_{m \in \bbn} \subset G$ such that $h_m^* \to \bbI_A$ as $m \to \infty$ and $\la h_m^*,h \ra = 0$ for each $m \in \bbn$.
\end{enumerate}
Then $(K,G)$ is approximately generated by an unconditional Schauder basis.
\end{proposition}

\begin{proof}
Setting $V_n := L^2(\calg_n)$ and
\begin{align*}
G_n := \{ \bbI_A / \mu(A) : A \in \cala_{\mu}(\calg_n) \}
\end{align*}
for each $n \in \bbn$, this is an immediate consequence of Proposition \ref{prop-self-dual-approx-Schauder}.
\end{proof}

\begin{corollary}\label{cor-prob-space-Schauder}
If $(X,\calx,\mu)$ is a separable probability space, then $(K,K)$ is approximately generated by an unconditional Schauder basis.
\end{corollary}

\begin{proof}
This is an immediate consequence of Propositions \ref{prop-prob-space-separable} and \ref{prop-L2-Schauder}.
\end{proof}

\begin{lemma}\label{lemma-approx-C-c-infty}
Suppose that $X \subset \bbr^d$ is an open subset equipped with its Borel $\sigma$-algebra $\calx = \calb(X)$. Let $B \in \calx$ be a Borel set such that $\mu(\partial B) = 0$. Then there exists a sequence $(\varphi_n)_{n \in \bbn} \subset K \cap C_c^{\infty}(X)$ with $\supp(\varphi_n) \subset \Int \, B$ for each $n \in \bbn$ such that $\varphi_n \to \bbI_B$ in $H$.
\end{lemma}

\begin{proof}
We define the open subset $U := \Int \, B$. By assumption we have $\mu(U) = \mu(B)$, and hence $\bbI_U = \bbI_B$ in $H$. Furthermore, we define the sequence of open subsets $(U_n)_{n \in \bbn}$ as
\begin{align}\label{U-n-def}
U_n := \bigg\{ x \in \bbr^d :d_{U^c}(x) > \frac{1}{n} \bigg\}, \quad n \in \bbn.
\end{align}
Then we have $U_n \uparrow U$, and hence $\bbI_{U_n} \to \bbI_U$ in $H$. Let $(\rho_m)_{m \in \bbn}$ be a sequence of mollifiers; that is, for each $m \in \bbn$ we have 
\begin{align*}
\rho_m \geq 0, \quad \rho_m \in C_c^{\infty}(\bbr^d), \quad \supp(\rho_m) \subset \overline{B(0,1/m)}\quad \text{and} \quad \int_{\bbr^d} \rho_m(x) dx = 1. 
\end{align*}
Now, let $\epsilon > 0$ be arbitrary. There exists $n \in \bbn$ such that
\begin{align*}
\| \bbI_U - \bbI_{U_n} \|_H < \frac{\epsilon}{2}. 
\end{align*}
Furthermore, by \cite[Thm. 4.22]{Brezis} there exists $m \geq n$ such that
\begin{align*}
\| \bbI_{U_n} - \rho_m \star \bbI_{U_n} \|_{L^2(\bbr^d,\mu)} < \frac{\epsilon}{2},
\end{align*}
where $\rho_m \star \bbI_{U_n} \geq 0$ denotes the convolution product of $\rho_m$ and $\bbI_{U_n}$. By \cite[Prop. 4.20]{Brezis} we have $\rho_m \star \bbI_{U_n} \in C^{\infty}(\bbr^d)$. Moreover, recalling the definition (\ref{U-n-def}), by \cite[Prop. 4.18]{Brezis} we have
\begin{align*}
\supp (\rho_m \star \bbI_{U_n}) \subset \overline{\supp (\rho_m) + \supp (\bbI_{U_n})} \subset \overline{B(0,1/m) + \supp (\bbI_{U_n})} \subset U.
\end{align*}
Therefore, by identification we have $\varphi := \rho_m \star \bbI_{U_n} \in K \cap C_c^{\infty}(X)$, and since $\bbI_U = \bbI_B$ in $H$, we obtain
\begin{align*}
\| \bbI_B - \varphi \|_H < \epsilon, 
\end{align*}
completing the proof.
\end{proof}

For what follows, we denote by $\lambda$ the Lebesgue measure.

\begin{proposition}\label{prop-L2-O-Schauder}
Suppose that $X \subset \bbr^d$ is an open subset equipped with its Borel $\sigma$-algebra $\calx = \calb(X)$, and that the measure $\mu$ is $\sigma$-finite such that $\mu \ll \lambda$ with $\frac{d \mu}{d \lambda} \in L_{\loc}^1(X,\lambda)$. Then the cone $(K,G)$ is approximately generated by an unconditional Schauder basis, where $G := K \cap C_c^{\infty}(X)$.
\end{proposition}

\begin{proof}
Let $n \in \bbn$ be arbitrary. For each $i \in \bbz^d$ we set
\begin{align*}
A_{n,i} := [2^{-n} i_1 , 2^{-n} (i_1 + 1) ) \times \ldots \times [ 2^{-n} i_d , 2^{-n} (i_d + 1) ),
\end{align*}
and we define the finite index set
\begin{align*}
I_n := \{ i \in \{ -n,\ldots,n-1 \}^d : A_{n,i} \subset X \}.
\end{align*}
Furthermore, we define the finite filtration $(\calg_n)_{n \in \bbn}$ as
\begin{align*}
\calg_n := \sigma(A_{n,i} : i \in I_n), \quad n \in \bbn.
\end{align*}
Then we have $\calb(X) = \sigma(\bigcup_{n \in \bbn} \calg_n)$. Now, we define the sequence $(B_n)_{n \in \bbn}$ as
\begin{align*}
B_n := \bigcup_{i \in I_n} A_{n,i} \in \calg_n \quad \text{for each $n \in \bbn$.}
\end{align*}
Then we have $B_n \uparrow X$. Moreover, since $\mu \ll \lambda$ with $\frac{d \mu}{d \lambda} \in L_{\loc}^1(X,\lambda)$, we have $\mu(B_n) < \infty$ for each $n \in \bbn$. Thus, by Proposition \ref{prop-conv-proj-sigma-fin}, for each $h \in H$ we have
\begin{align*}
\pi_{L^2(\calg_n)} h \to h \quad \text{in $H$.}
\end{align*}
Now, let $n \in \bbn$, $A \in \cala_{\mu}(\calg_n)$ and $h \in K$ with $\la \bbI_A,h \ra = 0$ be arbitrary. Since $\mu \ll \lambda$, we have $\mu(\partial A) = 0$. Hence, by Lemma \ref{lemma-approx-C-c-infty} there exists a sequence $(\varphi_{m})_{m \in \bbn} \subset G$ with $\supp (\varphi_{m}) \subset \Int \, A$ for each $m \in \bbn$ such that $\varphi_{m} \to \bbI_{A}$ in $H$ for $m \to \infty$. Furthermore, since $h \geq 0$ with $\la \bbI_A,h \ra = 0$, we have $\la \varphi_m,h \ra = 0$ for each $m \in \bbn$. Consequently, by Proposition \ref{prop-L2-Schauder} the cone $(K,G)$ is approximately generated by an unconditional Schauder basis.
\end{proof}

Now, we are ready to prove the following invariance result. Consider the stochastic framework from Section \ref{sec-jump-diffusion}.

\begin{theorem}\label{thm-G-star}
Suppose that $X \subset \bbr^d$ is an open subset equipped with its Borel $\sigma$-algebra $\calx = \calb(X)$, and that $\mu$ is a $\sigma$-finite measure such that $\mu \ll \lambda$ with $\frac{d \mu}{d \lambda} \in L_{\loc}^1(X,\lambda)$. We suppose that the following conditions are fulfilled:
\begin{enumerate}
\item The coefficients $(\alpha,\sigma,\gamma)$ are locally Lipschitz and satisfy the linear growth condition.

\item The semigroup $(S_t)_{t \geq 0}$ is pseudo-contractive.

\item The closed convex cone $K$ is invariant for the semigroup $(S_t)_{t \geq 0}$.

\item $A^*$ is a local operator relative to $(K,G)$, where $G := K \cap C_c^{\infty}(X)$.
\end{enumerate}
Then the following statements are equivalent:
\begin{enumerate}
\item[(i)] The closed convex cone $K$ is invariant for the SPDE (\ref{SPDE}).

\item[(ii)] We have (\ref{main-1}), and for all $(h^*,h) \in G \times K$ with $\langle h^*,h \rangle = 0$ we have (\ref{main-3}) and (\ref{main-4}).
\end{enumerate}
\end{theorem}

\begin{proof}
By Lemma \ref{lemma-local-operators} the semigroup $(S_t)_{t \geq 0}$ is a local semigroup relative to $(K,G)$. Therefore, taking into account Proposition \ref{prop-L2-O-Schauder}, the result follows from Theorems \ref{thm-general} and \ref{thm-general-2}.
\end{proof}

\section{Abstract $L^2$-spaces}\label{sec-abstract-L2}

In this section we investigate stochastic invariance of the positive cone in abstract $L^2$-spaces. First, we provide the required background about topological vector lattices. For further details, we refer, for example, to \cite[Chap. V]{Schaefer}.

Let $V$ be an $\bbr$-vector space. Furthermore, let $\leq$ be a binary relation over $V$ which is reflexive, anti-symmetric and transitive; more precisely:
\begin{itemize}
\item We have $x \leq x$ for all $x \in V$.

\item If $x \leq y$ and $y \leq x$, then we have $x = y$.

\item If $x \leq y$ and $y \leq z$, then we have $x \leq z$.
\end{itemize}
Then $(V,\leq)$ is called an \emph{ordered vector space} if the following axioms are satisfied:
\begin{enumerate}
\item If $x \leq y$, then we have $x+z \leq y+z$ for all $x,y,z \in V$.

\item If $x \leq y$, then we have $\alpha x \leq \alpha y$ for all $x,y \in V$ and $\alpha > 0$.
\end{enumerate}
Let $V$ be a topological vector space such that $(V,\leq)$ is an ordered vector space. Then we call $(V,\leq)$ an \emph{ordered topological vector space} if the positive cone
\begin{align*}
V_+ := \{ x \in V : x \geq 0 \} 
\end{align*}
is closed in $V$. In this case $V_+$ is a closed convex cone. A \emph{vector lattice} (or a \emph{Riesz space}) is an ordered vector space $(V,\leq)$ such that the supremum $x \vee y$ and the infimum $x \wedge y$ exist for all $x,y \in V$. We introduce further lattice operations. Namely, for $x \in V$ we define the positive part $x^+ := x \vee 0$, the negative part $x^- := -x \vee 0$, and the absolute value $|x| := x \vee (-x)$.

A vector lattice $V$ equipped with a norm $\| \cdot \|$ such that $(V,\| \cdot \|)$ is a Banach space and for all $x,y \in V$ with $|x| \leq |y|$ we have $\| x \| \leq \| y \|$ is called a \emph{Banach lattice}. Note that every Banach lattice is a topological vector lattice. A Banach lattice $V$ is called \emph{separable} if the Banach space $(V,\| \cdot \|)$ is separable. 

A Banach lattice $H$ is called a \emph{Hilbert lattice} if the Banach space $(H,\| \cdot \|)$ is a Hilbert space. For two Hilbert lattices $\bbh$ and $H$ a linear isomorphism $T \in L(\bbh,H)$ is called a \emph{lattice isomorphism} if $T \bbk = K$, where $\bbk := \bbh_+$ and $K := H_+$; cf. \cite[Thm. 9.17]{Aliprantis-Border}.

\begin{lemma}\label{lemma-abstract-L2}
For a Hilbert lattice $H$ the following statements are equivalent:
\begin{enumerate}
\item[(i)] We have $\| h + g \|^2 = \| h \|^2 + \| g \|^2$ for all $h,g \in H_+$ with $h \wedge g = 0$.

\item[(ii)] For all $h,g \in H_+$ we have $h \perp g$ if and only if $h \wedge g = 0$.

\item[(iii)] The closed convex cone $H_+$ is self-dual.
\end{enumerate}
\end{lemma}

\begin{proof}
(i) $\Leftrightarrow$ (ii): This equivalence is obvious.

\noindent(iii) $\Rightarrow$ (ii): This is a consequence of \cite[Lemma II.3]{Penney}.

\noindent(i) $\Rightarrow$ (iii): Note that $H$ is an abstract $L_2$-space in the sense of \cite[Def. 1]{Marti}, because every Hilbert space is weakly sequentially complete. Hence, according to \cite[Thm. 11]{Marti} there exist a measure space $(X,\calx,\mu)$ and an isometric lattice isomorphism $R \in L(H,\bbh)$, where $\bbh := L^2(X,\calx,\mu)$. Note that $\bbk = R K$, where $K := H_+$ and $\bbk := \bbh_+$, and that the closed convex cone $\bbk$ is self-dual. Hence, by Lemma \ref{lemma-transform-self-dual} the cone $K$ is also self-dual. 
\end{proof}

\begin{definition}
A Hilbert lattice $H$ is called an \emph{abstract $L^2$-space} if any of the equivalent conditions from Lemma \ref{lemma-abstract-L2} is fulfilled.
\end{definition}

\begin{example}
For every measure space $(X,\calx,\mu)$ the associated Hilbert space $H := L^2(X,\calx,\mu)$ is an abstract $L^2$-space.
\end{example}

\begin{lemma}\label{lemma-lattice-isom}
Let $\bbh$ and $H$ be two abstract $L^2$-spaces, and let $T \in L(\bbh,H)$ be a lattice isomorphism. Then the adjoint operator $T^* \in L(H,\bbh)$ is also a lattice isomorphism.
\end{lemma}

\begin{proof}
Let $h \in H_+$ be arbitrary. Furthermore, let $g \in \bbh_+$ be arbitrary. Then we have $Tg \in H_+$, and hence
\begin{align*}
\la T^* h,g \ra_{\bbh} = \la h,Tg \ra_H \geq 0,
\end{align*}
showing that $T^* h \in \bbh_+$. Now, let $g \in \bbh_+$ be arbitrary. Furthermore, let $h \in H_+$ be arbitrary. Then we have $T^{-1} h \in \bbh_+$, and hence, by Lemma \ref{lemma-adj-isom} we obtain
\begin{align*}
\la (T^*)^{-1} g,h \ra_H = \la (T^{-1})^* g,h \ra_H = \la g,T^{-1} h \ra_{\bbh} \geq 0,
\end{align*}
showing that $(T^*)^{-1} g \in H_+$.
\end{proof}

For what follows, let $H$ be a separable abstract $L^2$-space, and define the self-dual cone $K := H_+$.

\begin{proposition}\label{prop-abstract-Schauder}
The self-dual cone $(K,K)$ is approximately generated by an unconditional Schauder basis.
\end{proposition}

\begin{proof}
Without loss of generality, we may assume that $H \neq \{ 0 \}$. According to \cite[Cor. II.4]{Penney} there exist a finite measure space $(X,\calx,\mu)$ and a lattice isomorphism $R \in L(H,\bbh)$, where $\bbh := L^2(X,\calx,\mu)$. Due to Lemma \ref{lemma-proj-fin-measure-space} we may assume that $(X,\calx,\mu)$ is a probability space. Note that the Hilbert space $\bbh$ is separable, because $H$ is separable. Setting $\bbk := \bbh_+$, by Corollary \ref{cor-prob-space-Schauder} the cone $(\bbk,\bbk)$ is approximately generated by an unconditional Schauder basis. Now, we set $T := R^{-1}$ and $S := (T^{-1})^* = (T^*)^{-1}$. Taking into account Lemma \ref{lemma-lattice-isom}, the operators $T$ and $S$ are also lattice isomorphisms. Hence, by Proposition \ref{prop-seq-trans} it follows that $(K,K) = (T \bbk, S \bbk)$ is approximately generated by an unconditional Schauder basis.
\end{proof}

Now, we are ready to prove the following invariance result. Consider the stochastic framework from Section \ref{sec-jump-diffusion}.

\begin{theorem}\label{thm-abstract-L2}
We suppose that the following conditions are fulfilled:
\begin{enumerate}
\item The coefficients $(\alpha,\sigma,\gamma)$ are locally Lipschitz and satisfy the linear growth condition.

\item The semigroup $(S_t)_{t \geq 0}$ is pseudo-contractive.

\item The closed convex cone $K$ is invariant for the semigroup $(S_t)_{t \geq 0}$.
\end{enumerate}
If we have (\ref{main-1}), and for all $(h^*,h) \in K \times K$ with $\langle h^*,h \rangle = 0$ we have (\ref{main-3}) and (\ref{main-4}), then the closed convex cone $K$ is invariant for the SPDE (\ref{SPDE}).
\end{theorem}

\begin{proof}
In view of Proposition \ref{prop-abstract-Schauder}, this is a consequence of Theorem \ref{thm-general}.
\end{proof}

\section{Examples of Hilbert spaces and semigroups}\label{sec-Hilbert-semigroup}

In this section we present several examples of separable Hilbert spaces and semigroups thereon.

For any subset $X \subset \bbr^d$ and any nonnegative measurable function $w : X \to \bbr_+$ we define the Hilbert space $L_w^2(X) := L^2(X,\calb(X),w(x)dx)$. According to Lemma \ref{lemma-metric-separable} the Hilbert space $L_w^2(X)$ is separable. If $w \equiv 1$, then we simply write $L^2(X)$.

Now, let $v \in \bbr^d$ be such that $x + t v \in X$ for all $x \in X$ and all $t \geq 0$. We call a continuous function $w : X \to (0,\infty)$ an \emph{admissible weight function} if for some $\beta \in \bbr$ we have
\begin{align*}
w(x) \leq e^{\beta t} w(x+tv) \quad \text{for all $x \in X$ and all $t \geq 0$.}
\end{align*}
In most of the situations we will have $X = \bbr_+$ and $v=1$. We define the separable Hilbert space $H := L_w^2(X)$ and the self-dual closed convex cone $K := H_+$. Let $(S_t)_{t \geq 0}$ be the translation semigroup, which is defined as
\begin{align}\label{trans-def}
S_t h(x) = h(x+tv), \quad t \geq 0
\end{align}
for each $h \in H$. 

\begin{lemma}\label{lemma-L2-admissible-w}
Suppose that $w$ is an admissible weight function. Then the following statements are true:
\begin{enumerate}
\item $(S_t)_{t \geq 0}$ is a pseudo-contractive $C_0$-semigroup on $H$.

\item $K$ is invariant for the semigroup $(S_t)_{t \geq 0}$.
\end{enumerate}
\end{lemma}

\begin{proof}
The first statement is straightforward to check (cf. \cite[Lemma 4.6]{Desch} and its proof for the cases $I = \bbr$ and $I = \bbr_+$). Taking into account the definition (\ref{trans-def}), we see that for each $h \geq 0$ we have $S_t h \geq 0$, $t \geq 0$, proving the second statement.
\end{proof}

We will also consider translation semigroups on generalized Filipovi\'{c} spaces; see Appendix \ref{app-Filipovic-space} for further details.

Now, we consider the separable Hilbert space $H := L^2(\bbr^d)$, and the self-dual closed convex cone $K := H_+$. Let $(S_t)_{t \geq 0}$ be the heat semigroup, which is defined as
\begin{align}\label{heat-def-1}
S_0 h &:= h,
\\ \label{heat-def-2} S_t h &:= h \star p_t, \quad t > 0,
\end{align}
for each $h \in H$. Here, for each $t > 0$ we denote by $p_t : \bbr^d \to (0,\infty)$ the density of the centered multi-dimensional Normal distribution with covariance matrix $2t \Id$, which is given by
\begin{align*}
p_t(x) = \frac{1}{(4 \pi t)^{d/2}} \exp \bigg( - \frac{|x|^2}{4t} \bigg), \quad x \in \bbr^d,
\end{align*}
and we denote by $h \star p_t$ the convolution product of $h$ and $p_t$.

\begin{lemma}\label{lemma-Laplace-Rd}
The following statements are true:
\begin{enumerate}
\item $(S_t)_{t \geq 0}$ is a $C_0$-semigroup of contractions on $H$.

\item Its generator is given by $A = \Delta$ on the domain $\cald(A) = W^2(\bbr^d)$.

\item $K$ is invariant for the semigroup $(S_t)_{t \geq 0}$.

\item $A^*$ is a local operator relative to $(K,G)$, where $G := K \cap C_c^{\infty}(\bbr^d)$.
\end{enumerate}
\end{lemma}

\begin{proof}
The first two statements are well-known. Taking into account the definitions (\ref{heat-def-1}) and (\ref{heat-def-2}), we see that for each $h \geq 0$ we have $S_t h \geq 0$, $t \geq 0$, proving the third statement. The Laplace operator $\Delta$ is self-adjoint; that is $A^* = \Delta$. Therefore, we have $G \subset \cald(A^*)$. Let $(h^*,h) \in G \times K$ be such that $\la h^*,h \ra = 0$. Then we have $\supp(\Delta h^*) \subset \supp(h^*)$, and hence, noting that $h^*,h \geq 0$ we obtain
\begin{align*}
\la A^* h^*,h \ra = \la \Delta h^*,h \ra = 0,
\end{align*}
proving the last statement.
\end{proof}

Now, let $O \subset \bbr^d$ be an open, bounded subset with smooth boundary. We consider the separable Hilbert space $H := L^2(O)$ and the closed convex cone $K := H_+$. Let $A = \Delta$ be the Laplace operator defined on the domain
\begin{align*}
\cald(A) = W^2(O) \cap H_0^1(O).
\end{align*}
Hence, we consider Dirichlet boundary conditions.

\begin{lemma}\label{lemma-Laplace-domain}
The following statements are true:
\begin{enumerate}
\item $A$ generates a pseudo-contractive $C_0$-semigroup $(S_t)_{t \geq 0}$ on $H$.

\item $K$ is invariant for the semigroup $(S_t)_{t \geq 0}$.

\item $A^*$ is a local operator relative to $(K,G)$, where $G := K \cap C_c^{\infty}(O)$.
\end{enumerate}
\end{lemma}

\begin{proof}
By \cite[page 209]{Pazy} the operator $-\Delta$ is strongly elliptic, and we have
\begin{align*}
\langle -\Delta h,h \rangle_{L^2} = \| h \|_{W^2}^2 - \| h \|_{L^2}^2
\end{align*}
for all $h \in W^2(O) \cap H_0^1(O)$. Hence, according to \cite[Thm. 7.2.5]{Pazy} the operator $\Delta - \Id$ generates a $C_0$-semigroup $(T_t)_{t \geq 0}$ of contractions on $H$. Therefore, by Lemma \ref{lemma-generator-transform} it follows that $\Delta$ generates a pseudo-contractive $C_0$-semigroup $(S_t)_{t \geq 0}$; that is, there is a constant $\beta \geq 0$ such that
\begin{align*}
\| S_t \| \leq e^{\beta t} \quad \text{for all $t \in \bbr_+$.}
\end{align*}
This proves the first statement. For the proof of the second statement, we will show that $R_{\lambda} K \subset K$ for all $\lambda > \beta$. For this purpose, let $h \in K$ and $\lambda > \beta$ be arbitrary. We define the resolvent $u := R_{\lambda}h = (\lambda - \Delta)^{-1}h$. Since $u \in W^2(O) \cap H_0^1(O)$, it is the solution to the Dirichlet problem
\begin{align*}
\left\{
\begin{array}{rcll}
\Delta u - \lambda u & = & -h & \text{in $O$,}
\\ u & = & 0 & \text{on $\partial O$.}
\end{array}
\right.
\end{align*}
Let $\xi \in O$ be arbitrary, and let $W$ be an $\bbr^d$-valued standard Wiener process with $W_0 = \xi$ on some stochastic basis $(\Omega,\calf,(\calf_t)_{t \in \bbr_+},\bbp^{\xi})$. Then the first exit time
\begin{align*}
T_O := \inf \{ t \in \bbr_+ : W_t \notin O \}
\end{align*}
is a stopping time. Furthermore, since $O$ is open and bounded, by \cite[Lemma 5.7.4]{Karatzas-Shreve} we have $\bbe^{\xi}[T_O] < \infty$, and hence, by \cite[Prop. 5.7.2]{Karatzas-Shreve} we obtain the stochastic representation
\begin{align*}
u(\xi) = \frac{1}{2} \bbe^{\xi} \bigg[ \int_0^{T_O} \exp \bigg( -\frac{\lambda}{2} \int_0^t W_s ds \bigg) h(W_t) dt \bigg] \geq 0.
\end{align*}
Therefore, by Lemma \ref{lemma-cone-semi-inv} the cone $K$ is invariant for the semigroup $(S_t)_{t \geq 0}$, proving the second statement. The proof of the last statement is analogous to that of Lemma \ref{lemma-Laplace-Rd}.
\end{proof}

\section{Applications}\label{sec-applications}

In this section we present several examples of SPDEs arising in natural sciences and economics. In all of the following examples, it is tacitly assumed that the coefficients $(\alpha,\sigma,\gamma)$ of the SPDE (\ref{SPDE}) are locally Lipschitz and satisfy the linear growth condition.

Our first example is the stochastic cable equation (cf. \cite[Sec. 0.8]{Da_Prato})
\begin{align}\label{SPDE-cable}
\left\{
\begin{array}{rcl}
dv_t & = & \frac{1}{\tau} \big( \lambda^2 \frac{d^2}{d \xi^2} v_t - v_t \big) dt + \sigma(v_t) dW_t + \int_E \gamma(v_{t-},x) (N(dt,dx) - F(dx)dt) \medskip
\\ v_0 & = & h_0,
\end{array}
\right.
\end{align}
which describes the voltage of an electric cable over time. The constants $\lambda,\tau > 0$ are physical constants of the electric cable; $\lambda$ is the length constant and $\tau$ is the time constant. We will model the electric cable on the unit interval $[0,1]$ with Dirichlet boundary conditions, which means that there is no voltage at the end points of the cable. As the voltage should be a nonnegative quantity, we are interested in conditions which ensure that the voltage is nonnegative. Now, we choose the state space $H := L^2((0,1))$, the self-dual closed convex cone $K := H_+$ of nonnegative functions, and the generating system $G := K \cap C_c^{\infty}((0,1))$.

\begin{proposition}
The following statements are equivalent:
\begin{enumerate}
\item[(i)] The closed convex cone $K$ is invariant for the stochastic cable equation (\ref{SPDE-cable}).

\item[(ii)] We have (\ref{main-1}), and for all $(h^*,h) \in G \times K$ with $\langle h^*,h \rangle = 0$ we have (\ref{main-3}) and (\ref{main-4}).
\end{enumerate}
\end{proposition}

\begin{proof}
Taking into account Lemma \ref{lemma-Laplace-domain} and Lemma \ref{lemma-generator-transform}, this is a consequence of Theorem \ref{thm-G-star}.
\end{proof}

Next, we consider the stochastic heat equation
\begin{align}\label{SPDE-heat}
\left\{
\begin{array}{rcl}
du_t & = & ( a \Delta u_t + \alpha(u_t) ) dt + \sigma(u_t) dW_t + \int_E \gamma(u_{t-},x) (N(dt,dx) - F(dx)dt) \medskip
\\ u_0 & = & h_0,
\end{array}
\right.
\end{align}
which describes the heat of a medium in a region $O \subset \bbr^d$ over time, where $a > 0$ is the heat conductivity. We are interested in conditions which ensure that the temperature is nonnegative in the region. This means that the temperature does not fall below the freezing point. The region can be $O = \bbr^d$, or an open, bounded subset $O \subset \bbr^d$ with smooth boundary, and we consider Dirichlet boundary conditions. Hence, we choose the state space $H := L^2(O)$, the self-dual closed convex cone $K := H_+$ of nonnegative functions, and the generating system $G := K \cap C_c^{\infty}(O)$.

\begin{proposition}\label{prop-heat}
The following statements are equivalent:
\begin{enumerate}
\item[(i)] The closed convex cone $K$ is invariant for the stochastic heat equation (\ref{SPDE-heat}).

\item[(ii)] We have (\ref{main-1}), and for all $(h^*,h) \in G \times K$ with $\langle h^*,h \rangle = 0$ we have (\ref{main-3}) and (\ref{main-4}).
\end{enumerate}
\end{proposition}

\begin{proof}
Taking into account Lemmas \ref{lemma-Laplace-Rd}, \ref{lemma-Laplace-domain} and Lemma \ref{lemma-generator-transform}, this is a consequence of Theorem \ref{thm-G-star}.
\end{proof}

\begin{remark}[Parabolic Anderson model]
Consider the particular situation, where $\alpha = 0$, the volatility $\sigma$ is given by $\sigma^j(h) = \lambda_j h$ for $j \in \bbn$ and $h \in H$ with a sequence $\lambda = (\lambda_j)_{j \in \bbn} \in \ell^2(\bbn)$, and $\gamma = 0$. This may be seen as a version of the parabolic Anderson model; see, for example \cite{Anderson}. According to Proposition \ref{prop-heat}, in this case the closed convex cone $K$ is invariant for the stochastic heat equation (\ref{SPDE-heat}).
\end{remark}

Now, we consider the Heath-Jarrow-Morton-Musiela (HJMM) equation from mathematical finance. This SPDE models the term structure of interest rates in a market of zero coupon bonds. We recall that a zero coupon bond with maturity $T$ is a financial asset that pays to the holder one monetary unit at time $T$. Its price at $t \leq T$ can be written as the continuous discounting of one unit of the domestic currency
\begin{align*}
P(t,T) = \exp \bigg( -\int_t^T f(t,s)ds \bigg),
\end{align*}
where $f(t,T)$ denote the so-called forward rates. After transforming the original Heath-Jarrow-Morton (HJM) dynamics of the forward rates (see \cite{HJM}) by means of the Musiela parametrization $r_t(\xi) = f(t,t+\xi)$ (see \cite{Musiela} or \cite{Brace-Musiela}), we arrive at the HJMM equation
\begin{align}\label{SPDE-HJMM}
\left\{
\begin{array}{rcl}
dr_t & = & \big( \frac{d}{d\xi} r_t + \alpha(r_t) \big) dt + \sigma(r_t) dW_t + \int_E \gamma(r_{t-},x) (N(dt,dx) - F(dx)dt) \medskip
\\ r_0 & = & h_0.
\end{array}
\right.
\end{align}
We choose the state space $H := L_w^2(\bbr_+)$ with an admissible weight function $w : \bbr_+ \to (0,\infty)$. Although nowadays interest rates can also be slightly negative, in principle it is a desirable feature to have nonnegative interest rate curves. Thus, we introduce the self-dual closed convex cone $K := H_+$ of nonnegative functions.

\begin{proposition}
If we have (\ref{main-1}), and for all $(h^*,h) \in K \times K$ with $\langle h^*,h \rangle = 0$ we have (\ref{main-3}) and (\ref{main-4}), then the closed convex cone $K$ is invariant for the HJMM equation (\ref{SPDE-HJMM}).
\end{proposition}

\begin{proof}
Taking into account Lemma \ref{lemma-L2-admissible-w}, this is a consequence of Theorem \ref{thm-abstract-L2}.
\end{proof}

Typically, the drift $\alpha : H \to H$ in the HJMM equation (\ref{SPDE-HJMM}) cannot be chosen freely. This is due to no-arbitrage considerations in the bond market. In the spirit of the Benchmark Approach \cite{Platen}, we model the forward rate dynamics under the real-word probability measure. Then, due to aforementioned no-arbitrage constraints, the drift is given by
\begin{align}\label{drift-HJM}
\alpha(h) = \sum_{j=1}^{\infty} \sigma^j(h) ( \theta^j(h) + \Sigma^j(h) ) - \int_E \gamma(h,x) \big( \exp( -\phi(h,x) - \Gamma(h,x) ) - 1 \big) F(dx),
\end{align}
see, for example \cite{Platen-Christensen} or \cite{B-N-Platen}. Here $\theta : H \to L_0^2(\bbr)$ and $\phi : H \times E \to \bbr$ are further mappings, which represent the \emph{market prices of risks}, and we agree on the notation
\begin{align*}
\Sigma^j(h) := \int_0^{\bullet} \sigma^j(h)(\eta) d\eta \quad \text{and} \quad \Gamma(h,x) &:= \int_0^{\bullet} \gamma(h,x)(\eta) d\eta.
\end{align*}
This framework includes the well-known HJMM equation, considered under a risk-neutral probability measure; see, for example \cite{fillnm, FTT-positivity}. 

\begin{proposition}
Suppose we have (\ref{main-1}), and that for all $(h^*,h) \in K \times K$ with $\langle h^*,h \rangle = 0$ we have (\ref{main-4}). Then the following statements are equivalent:
\begin{enumerate}
\item[(i)] For all $(h^*,h) \in K \times K$ with $\langle h^*,h \rangle = 0$ we have (\ref{main-3}).

\item[(ii)] For all $(h^*,h) \in K \times K$ with $\langle h^*,h \rangle = 0$ we have (\ref{main-5}).
\end{enumerate}
\end{proposition}

\begin{proof}
Let $(h^*,h) \in K \times K$ with $\langle h^*,h \rangle = 0$ be arbitrary. Using (\ref{main-1}), for all $(h^*,h) \in K \times K$ with $\langle h^*,h \rangle = 0$ we have
\begin{align*}
\la h^*,\gamma(h,x) \ra = \la h^*,h + \gamma(h,x) \ra \geq 0 \quad \text{for $F$-almost all $x \in E$.}
\end{align*}
Furthermore, by the structure (\ref{drift-HJM}) of the drift $\alpha$, and by (\ref{main-4}), we obtain
\begin{align*}
&\la h^*,\alpha(h) \ra - \int_E \la h^*,\gamma(h,x) \ra F(dx)
\\ &= \sum_{j=1}^{\infty} \big\la h^*, \sigma^j(h) ( \theta^j(h) + \Sigma^j(h) ) \big\ra - \int_E \big\la h^*, \gamma(h,x) \exp( -\phi(h,x) - \Gamma(h,x) ) \big\ra F(dx)
\\ &= - \int_E \big\la h^*, \gamma(h,x) \exp( -\phi(h,x) - \Gamma(h,x) ) \big\ra F(dx),
\end{align*}
which provides the stated equivalence.
\end{proof}

Our next example is in the field of actuarial mathematics. As in \cite{Tappe-Weber} we consider survival probabilities
\begin{align*}
G(t,T,x) = \bbp(\tau^x > T | \calf_t),
\end{align*}
where $\tau^x$ is the time of death of an individual born at date $-x$. Apparently, by the definition of $G(t,T,x)$ we must have
\begin{align}\label{mortality-cond}
G(t,S,x) &\geq G(t,T,x), \quad \text{if $S \leq T$.}
\end{align}
As in \cite{Tappe-Weber}, we define the survival probabilities as
\begin{align*}
G(t,T,x) := G(0,0,x) \exp \bigg( -\int_{-x}^t \mu(s,s,x) ds \bigg) \exp \bigg( -\int_{-x \vee t}^T \mu(t,u,x) du \bigg),
\end{align*}
where $\mu(t,T,x)$ denote the mortality rates. Performing the Musiela type transformation $m_t(s,y) = \mu(t,s+t,y-t)$, we obtain
\begin{equation}\label{mortality-G}
\begin{aligned}
G(t,T,x) &= G(0,0,x) \exp \bigg( -\int_{-x}^t m_s(0,x+s) ds \bigg)
\\ &\quad \times \exp \bigg( -\int_{-x \vee t}^T m_t(u-t,x+t) du \bigg),
\end{aligned}
\end{equation}
and we arrive at an SPDE of the type
\begin{align}\label{SPDE-mortality}
\left\{
\begin{array}{rcl}
dm_t & = & \big( (\partial_s - \partial_y) m_t + \alpha(m_t) \big) dt + \sigma(m_t) dW_t
\\ & & + \int_E \gamma(m_{t-},x) (N(dt,dx) - F(dx)dt) \medskip
\\ m_0 & = & h_0.
\end{array}
\right.
\end{align}
We choose the space $H := L_w^2(\Xi)$ for an admissible weight function $w : \Xi \to (0,\infty)$ as state space, where the domain $\Xi \subset \bbr^2$ is defined as
\begin{align*}
\Xi := \{ (s,y) \in \bbr_+ \times \bbr : -y \leq s \}. 
\end{align*}
According to Lemma \ref{lemma-L2-admissible-w}, the translation semigroup $(S_t)_{t \geq 0}$ given by
\begin{align*}
S_t h(s,y) := h(s+t,y-t), \quad t \geq 0
\end{align*}
for each $h \in H$, is a pseudo-contractive $C_0$-semigroup on $H$, and the self-dual closed convex cone $K := H_+$ is invariant for the semigroup $(S_t)_{t \geq 0}$. If $K$ is invariant for the mortality rate equation (\ref{SPDE-mortality}), then the mortality rates are nonnegative, and by (\ref{mortality-G}) condition (\ref{mortality-cond}) is fulfilled.

\begin{proposition}
If we have (\ref{main-1}), and for all $(h^*,h) \in K \times K$ with $\langle h^*,h \rangle = 0$ we have (\ref{main-3}) and (\ref{main-4}), then the closed convex cone $K$ is invariant for the mortality rate equation (\ref{SPDE-mortality}).
\end{proposition}

For insurance companies it may be interesting to consider joint models, say for zero coupon bonds and mortality rates. This leads to an SPDE of the type
\begin{align}\label{SPDE-hybrid}
\left\{
\begin{array}{rcl}
dR_t & = & \big( A R_t + \alpha(R_t) \big) dt + \sigma(R_t) dW_t + \int_E \gamma(R_{t-},x) (N(dt,dx) - F(dx)dt) \medskip
\\ R_0 & = & h_0.
\end{array}
\right.
\end{align}
for the process $R = (r,m)$ on the product space $\calh := H_1 \times H_2$, where $H_1 := L_w^2(\bbr_+)$ and $H_2 := L_v^2(\Xi)$ with admissible weight functions $w : \bbr_+ \to (0,\infty)$ and $v : \Xi \to (0,\infty)$, and where the generator is given by $A := (d/d\xi, \partial_s - \partial_y)$. According to our previous discussion, we consider the closed convex cone $\calk := K_1 \times K_2$, where $K_1 := (H_1)_+$ and $K_2 := (H_2)_+$.

\begin{proposition}
Suppose we have (\ref{main-1-prod}), for all $i=1,2$ and $(h^*,h) \in K_i \times \calk$ with $\la h^*,h_i \ra_{H_i} = 0$ we have (\ref{main-3-prod}) and (\ref{main-4-prod}). Then the closed convex cone $\calk$ is invariant for the hybrid SPDE (\ref{SPDE-hybrid}).
\end{proposition}

\begin{proof}
In view of Proposition \ref{prop-abstract-Schauder}, this is a consequence of Theorem \ref{thm-prod-cone}.
\end{proof}

Next, we consider an equation concerning the stochastic modeling of energy markets. For $0 \leq t \leq \tau < \infty$ we denote by $f(t,\tau)$ the futures price at time $t$ of electricity to be delivered constantly over a fixed period starting at $\tau$. These future prices are often modelled as $f(t,\tau) = \bbe_{\bbq}[S_{\tau} | \calf_t]$, where $\bbq \approx \bbp$ denotes a risk-neutral measure, and where $S$ is the spot price process, which is the price of electricity for immediate delivery; see, for example \cite{Benth}. After performing the Musiela type transformation $r_t(\xi) = f(t,t+\xi)$, we have an SPDE of the type
\begin{align}\label{SPDE-energy}
\left\{
\begin{array}{rcl}
dr_t & = & \frac{d}{d\xi} r_t dt + \sigma(r_t) dW_t + \int_E \gamma(r_{t-},x) (N(dt,dx) - F(dx)dt) \medskip
\\ r_0 & = & h_0.
\end{array}
\right.
\end{align}
Here we choose the state space $H := L_w^2(\bbr_+)$ with an admissible weight function $w : \bbr_+ \to (0,\infty)$. Of course, futures prices should be nonnegative. Hence, we consider the self-dual closed convex cone $K := H_+$ of nonnegative functions.

\begin{proposition}\label{prop-energy}
If we have (\ref{main-1}), and for all $(h^*,h) \in K \times K$ with $\langle h^*,h \rangle = 0$ we have (\ref{main-4}) and (\ref{main-5}), then the closed convex cone $K$ is invariant for the energy market equation (\ref{SPDE-energy}).
\end{proposition}

\begin{proof}
Taking into account Lemma \ref{lemma-L2-admissible-w}, this is a consequence of Theorem \ref{thm-abstract-L2} and Proposition \ref{prop-alpha-zero}.
\end{proof}

\begin{remark}
If $\sigma^j(h) = \lambda_j h$ for $j \in \bbn$ and $h \in H$ with a sequence $\lambda = (\lambda_j)_{j \in \bbn} \in \ell^2(\bbn)$, and $\gamma \equiv 0$, then the conditions from Proposition \ref{prop-energy} are fulfilled. Indeed, in this case the corresponding futures prices $f(t,\tau)$ are simply given by a stochastic exponential; a model, which is often used for energy markets.
\end{remark}

\begin{remark}
Let us just mention that we can derive similar invariance results when modelling commodity markets; see, for example \cite{Benth-Kruehner}.
\end{remark}

Now, we consider variance swap models; see, for example \cite{Buehler}. Here the primary security accounts are the quadratic variation of the logarithm of a stock; more precisely $V(t,T) = \bbe[ \la X,X \ra_T | \calf_t ]$ with a stock $S$ of the form $S = S_0 \cale(X)$, where $X$ is a semimartingale with $\Delta X > -1$, and where $\cale(X)$ denotes the stochastic exponential of $X$. After performing the Musiela type transformation $U(t,\xi) = V(t,t+\xi)$, we arrive at an SPDE of the form
\begin{align}\label{SPDE-variance}
\left\{
\begin{array}{rcl}
dU_t & = & \frac{d}{d\xi} U_t dt + \sigma(U_t) dW_t + \int_E \gamma(U_{t-},x) (N(dt,dx) - F(dx)dt) \medskip
\\ U_0 & = & h_0.
\end{array}
\right.
\end{align}
Here we choose the Filipovi\'{c} space $H := H_w := H_w(\bbr_+)$ with an admissible weight function $w : \bbr_+ \to [1,\infty)$ as state space; see Appendix \ref{app-Filipovic-space}. According to the specification above, the variance swaps should be nonnegative and increasing. Therefore, we consider the closed convex cone
\begin{align*}
K := \{ h \in H : h(0) \geq 0 \text{ and } h' \geq 0 \}.
\end{align*}
Due to Proposition \ref{prop-Fil-1-dim-semi}, the translation semigroup $(S_t)_{t \geq 0}$ is a pseudo-contractive $C_0$-semigroup on $H$. Furthermore, it is straightforward to check that $K$ is invariant for the semigroup $(S_t)_{t \geq 0}$. In the sequel, we will identify $\bbr$ with all constant functions and use the notation
\begin{align*}
H_w^0 := \{ h \in H_w : h(0) = 0 \}.
\end{align*}

\begin{proposition}\label{prop-variance-curve}
Suppose we have (\ref{main-1}), and that for all $(h^*,h) \in (\bbr \cup H_w^0) \cap K \times K$ with $\la h^*,h \ra = 0$ we have (\ref{main-4}) and (\ref{main-5}). Then the closed convex cone $K$ is invariant for the variance swap equation (\ref{SPDE-variance}).
\end{proposition}

\begin{proof}
According to Proposition \ref{prop-direct-sum-1} the subspaces $\bbr$ and $H_w^0$ are orthogonal closed subspaces, and we have the direct sum decomposition $H = \bbr \oplus H_w^0$. Furthermore, the cone $K$ admits the direct sum decomposition
\begin{align*}
K = \bbr_+ \oplus (H_w^0 \cap K) = (\bbr \cap K) \oplus (H_w^0 \cap K).
\end{align*}
Note that the linear mappings $T_1 \in L(\bbr)$ and $T_2 \in L(L^2(\bbr_+),H_w^0)$ given by
\begin{align*}
T_1 a = a \quad \text{and} \quad Tf = \int_0^{\bullet} f(x) w^{-\frac{1}{2}}(x) dx
\end{align*}
are isomorphic isomorphisms. By Proposition \ref{prop-abstract-Schauder} the self-dual closed convex cone $L_+^2(\bbr_+)$ is approximately generated by an unconditional Schauder basis. Furthermore, we have
\begin{align*}
T_1(\bbr_+) = \bbr \cap K \quad \text{and} \quad T_2(L_+^2(\bbr_+)) = H_w^0 \cap K.
\end{align*}
Therefore, the result is a consequence of Theorem \ref{thm-direct-sum} and Proposition \ref{prop-alpha-zero}.
\end{proof}

\begin{remark}
Suppose that the conditions from Proposition \ref{prop-variance-curve} are fulfilled. According to the specification above, it is reasonable to start the variance swap model with an initial curve $h_0 \in K$ such that $h_0(0) = 0$; that is $h_0 \in K \cap H_w^0$.
\end{remark}

\begin{remark}
In the aforementioned article \cite{Buehler} an SPDE for the forward variance curves $v(t,T) := \partial_T V(t,T)$ has been derived. Here we directly consider the SPDE (\ref{SPDE-variance}) for the variance swaps. 
\end{remark}

Now, we focus on financial models with credit risk. A defaultable $(T,x)$-bond with maturity $T$ and credit rating $x \in I \subset [0,1]$ is a financial contract which pays to the holder one monetary unit at time $T$, unless bankruptcy has happened until then. In this context, one also speaks about collateralized debt obligations (CDOs). The bonds are often modeled as
\begin{align*}
P(t,T,x) = \bbI_{\{ L_t \leq x \}} \exp \bigg( -\int_t^T f(t,u,x) du \bigg),
\end{align*}
where $f(t,T,x)$ denote the corresponding forward rates, and where $L$ denotes the loss process; see, for example \cite{FOS} for further details. Now, let us consider a CDO model with a finite set $I := \{ 0 \leq x_1 < x_2 < \ldots < x_m = 1 \}$ of credit ratings. After performing the Musiela type transformation $r_t(\xi,\eta) = f(t,t+\xi,\eta)$, we obtain an SPDE of the form
\begin{align}\label{SPDE-CDO}
\left\{
\begin{array}{rcl}
dr_t & = & \big( \frac{d}{d\xi} r_t + \alpha(r_t) \big) dt + \sigma(r_t) dW_t + \int_E \gamma(r_{t-},x) (N(dt,dx) - F(dx)dt) \medskip
\\ r_0 & = & h_0.
\end{array}
\right.
\end{align}
We choose the state space $\calh := H^m$, where $H := L_w^2(\bbr_+)$ for some admissible weight function $w : \bbr_+ \to (0,\infty)$. Similar models have been considered in \cite{Barski} and \cite{Schmidt-Tappe}. Since the positivity and the monotonicity of the forward rates is a desirable feature, we consider the closed convex cone
\begin{align*}
\calk := \{ h \in \calh : h_1 \geq \ldots \geq h_m \geq 0 \}.
\end{align*}
According to Lemma \ref{lemma-L2-admissible-w} the translation semigroup $(S_t)_{t \geq 0}$ is a pseudo-contractive $C_0$-semigroup on $\calh$. Furthermore, it is straightforward to check that the cone $\calk$ is invariant for the semigroup $(S_t)_{t \geq 0}$.

\begin{proposition}
If we have (\ref{main-1-prod}), for $(h^*,h) \in K \times \calk$ with $\la h^*,h_m \ra = 0$ we have
\begin{align*}
&\langle h^*,\alpha_m(h) \rangle - \int_E \langle h^*,\gamma_m(h,x)\rangle F(dx) \geq 0,
\\ &\langle h^*,\sigma_m^j(h) \rangle = 0, \quad j \in \bbn,
\end{align*}
and for all $i=2,\ldots,m$ and $(h^*,h) \in K \times \calk$ with $\langle h^*, h_i-h_{i-1} \rangle = 0$ we have
\begin{align*}
&\langle h^*,\alpha_{i-1}(h)-\alpha_i(h) \rangle - \int_E \langle h^*,\gamma_{i-1}(h,x) - \gamma_i(h,x) \rangle F(dx) \geq 0,
\\ &\langle h^*,\sigma_{i-1}^j(h)-\sigma_i^j(h) \rangle = 0, \quad j \in \bbn,
\end{align*}
then the closed convex cone $\calk$ is invariant for the CDO equation (\ref{SPDE-CDO}).
\end{proposition}

\begin{proof}
Note that the cone $\calk$ has the representation
\begin{align*}
\calk := \bigcap_{i=1}^{m-1} \{ h \in \calh : h_i - h_{i+1} \in K \} \cap \{ h \in \calh : h_m \in K \},
\end{align*}
where $K := H_+$. Hence, it is of the form (\ref{cone-matrix}) with the matrix $M \in \bbr^{m \times m}$ given by
\begin{align*}
M_{ik} =
\begin{cases}
1, & \text{if $i=k$,}
\\ -1, & \text{if $i=k-1$,}
\\ 0, & \text{otherwise,}
\end{cases}
\end{align*}
and $K_i = K$ for all $i=1,\ldots,m$. Therefore, the result is a consequence of Proposition \ref{prop-abstract-Schauder} and Theorem \ref{thm-inv-matrix}.
\end{proof}

\begin{remark}
Similar results can be found in \cite{Barski} and \cite{Schmidt-Tappe}.
\end{remark}

\begin{remark}\label{rem-multi-curve-2}
A similar situation arises when modeling multiple yield curves (see, for example \cite{Multi-Curve}). Then it is desirable to have spreads which are ordered with respect to different tenors. In this situation, we can also choose the state space $\calh := H^m$ and the closed convex cone
\begin{align*}
\calk := \{ h \in \calh : h_1 \leq \ldots \leq h_m \},
\end{align*}
or, if the spreads should additionally be nonnegative, the closed convex cone 
\begin{align*}
\calk := \{ h \in \calh : 0 \leq h_1 \leq \ldots \leq h_m \}.
\end{align*}
As already mentioned in Remark \ref{rem-multi-curve}, this situation is covered by Corollaries \ref{cor-mon-1} and \ref{cor-mon-2}.
\end{remark}

Now, we consider an alternative approach to CDO modeling; the so-called SPA model from \cite{SPA}. Here the state variable is the portfolio loss distribution. More precisely, let us introduce $p(t,T,x) := \bbp(\ell(T) \leq x | \calf_t)$ for $x \in I := [0,1]$, where $\ell$ denotes the loss process, which is an increasing process with values in $[0,1]$. As already noted in \cite{SPA}, by the definition of $p(t,T,x)$ we must have
\begin{align}\label{SPA-1}
p(0,0,x) &= 1,
\\ \label{SPA-2} p(t,T,1) &= 1,
\\ \label{SPA-3} p(t,S,x) &\geq p(t,T,x), \quad \text{if $S \leq T$,}
\\ \label{SPA-4} p(t,T,x) &\leq p(t,T,y), \quad \text{if $x \leq y$.}
\end{align}
As in \cite{SPA}, we define the loss distributions as
\begin{align*}
p(t,T,x) := \exp \bigg( - \int_0^t f(u,u,x) du - \int_t^T f(t,u,x) du \bigg),
\end{align*}
where $f(t,T,x)$ denote the corresponding forward rates. Then condition (\ref{SPA-1}) is satisfied. Performing the Musiela type transformation $r_t(\xi,\eta) := f(t,t+\xi,\eta)$, we obtain
\begin{align}\label{p-SPA}
p(t,T,x) = \exp \bigg( - \int_0^t r_u(0,x) du - \int_0^{T-t} r_t(\eta,x) d\eta \bigg),
\end{align}
and we arrive at an SPDE of the type
\begin{align}\label{SPDE-SPA}
\left\{
\begin{array}{rcl}
dr_t & = & \big( \frac{d}{d\xi} r_t + \alpha(r_t) \big) dt + \sigma(r_t) dW_t + \int_E \gamma(r_{t-},x) (N(dt,dx) - F(dx)dt) \medskip
\\ r_0 & = & h_0.
\end{array}
\right.
\end{align}
We choose the two-dimensional Filipovi\'{c} space $H := H_{w,1}(\bbr_+ \times [0,1])$, equipped with the norm $\| \cdot \|_{(w,1),(0,1)}$, where $w : \bbr_+ \to [1,\infty)$ denotes an admissible weight function, as state space; see Appendix \ref{app-Filipovic-space} for further details. According to Proposition \ref{prop-Fil-2-dim-semi} the translation semigroup $(S_t)_{t \geq 0}$ given by
\begin{align*}
S_t h(\xi,\eta) := h(\xi+t,\eta), \quad t \geq 0
\end{align*}
for each $h \in H$ is a pseudo-contractive $C_0$-semigroup on $H$. Furthermore, according to Proposition \ref{prop-Fil-2-dim-decomp} we have the direct sum decomposition
\begin{align}\label{decomp-SPA}
H = \bbr \oplus H_w^0(\bbr_+) \oplus H_1^1([0,1]) \oplus H_{w,1}^{(0,1)}(\bbr_+ \times [0,1]),
\end{align}
and the closed subspaces appearing in (\ref{decomp-SPA}) are mutual orthogonal. Note that the closed subspace
\begin{align*}
H_0 := H_1^1([0,1]) \oplus H_{w,1}^{(0,1)}(\bbr_+ \times [0,1]).
\end{align*}
is invariant for the semigroup $(S_t)_{t \geq 0}$. Therefore, we may and will consider the SPA equation (\ref{SPDE-SPA}) on the state space $H_0$ with coefficients $\alpha : H_0 \to H_0$, $\sigma : H_0 \to L_2^0(H_0)$ and $\gamma : H_0 \times E \to H_0$. According to Proposition \ref{prop-Fil-2-dim-semi-2} the semigroup $(S_t)_{t \geq 0}$ is even a $C_0$-semigroup of contractions on $H_0$. Note that $h(\cdot,1) = 0$ for each $h \in H_0$. Therefore, by (\ref{p-SPA}) condition (\ref{SPA-2}) is fulfilled. Now we consider the closed convex cone
\begin{align*}
K := \{ h \in H_0 : h_{\eta} \leq 0, h_{\xi \eta} \leq 0 \}.
\end{align*}
Then for each $h \in K$ we have $h \geq 0$ and
\begin{align*}
h(\xi,\eta_1) \geq h(\xi,\eta_2), \quad \text{if $\eta_1 \leq \eta_2$.}
\end{align*}
Therefore, if $K$ is invariant for the SPA equation (\ref{SPDE-SPA}), then by (\ref{p-SPA}) conditions (\ref{SPA-3}) and (\ref{SPA-4}) are fulfilled.

\begin{proposition}
Suppose we have (\ref{main-1}), and for all
\begin{align*}
(h^*,h) \in \big( H_1^1([0,1]) \cup H_{w,1}^{(0,1)}(\bbr_+ \times [0,1]) \big) \cap K \times K 
\end{align*}
with $\la h^*,h \ra = 0$ we have (\ref{main-3}) and (\ref{main-4}). Then the closed convex cone $K$ is invariant for the SPA equation (\ref{SPDE-SPA}).
\end{proposition}

\begin{proof}
Note that the cone $K$ admits the direct sum decomposition
\begin{align*}
K = ( H_1^1([0,1]) \cap K ) \oplus ( H_{w,1}^{(0,1)}(\bbr_+ \times [0,1]) \cap K ).
\end{align*}
Considering the isometric isomorphisms
\begin{align*}
T_1 \in L \big( L^2([0,1]),H_1^1([0,1]) \big) \quad \text{and} \quad T_2 \in L \big( L^2(\bbr_+ \times [0,1]),H_{w,1}^{(0,1)}(\bbr_+ \times [0,1]) \big)
\end{align*}
given by
\begin{align*}
Tf(y) &= \int_1^y f(\eta) d\eta,
\\ Tg(x,y) &= \int_0^x \int_1^y g(\xi,\eta) w^{-\frac{1}{2}} d\eta d\xi,
\end{align*}
we can proceed as in the proof of Proposition \ref{prop-variance-curve}.
\end{proof}

In our last example, we consider the FX-like (foreign exchange like) model from \cite{Krabichler} (see also the recent preprints \cite{Krabichler-Teichmann-1, Krabichler-Teichmann-2}), which includes credit and liquidity risk. In this framework we have a non-defaultable zero coupon bond $P(t,T)$ with $P(T,T) = 1$ and a defaultable zero coupon bond $\widetilde{P}(t,T)$ with $0 < \widetilde{P}(T,T) \leq 1$. The defaultable bond is given by $\widetilde{P}(t,T) := S_t Q(t,T)$, where $Q(t,T)$ denotes a synthetic non-defaultable zero coupon bond with $Q(T,T) = 1$, and where $S$ is a $(0,1]$-valued process such that $S_0 = 1$. The process $S$ is also called the recovery rate or spot FX rate, and for any $t \in \bbr_+$ the event $\{ S_t < 1 \}$ is called a liquidity squeeze at time $t$. We can express the recovery rate as $S = \exp(-X)$, where $X$ is an $\bbr_+$-valued process such that $X_0 = 0$, and then the event $\{ X_t > 0 \}$ is a liquidity squeeze at time $t$. According to the Jarrow \& Turnbull 1991 paradigm, the bonds $P(t,T)$ and $Q(t,T)$ are considered as non-defaultable zero coupon bonds in different currencies. In this setting, the monotonicity property $\widetilde{P}(t,T) \leq P(t,T)$ is a desirable feature. Note that this property is satisfied if 
\begin{align}\label{FX-cond-1}
Q(t,T) &\leq P(t,T),
\\ \label{FX-cond-2} X &\geq 0.
\end{align}
As in \cite[Sec. 6]{Krabichler}, we assume that the bonds are given by
\begin{align*}
P(t,T) &= \exp \bigg( -\int_t^T f^{\rm dom}(t,u) du \bigg),
\\ Q(t,T) &= \exp \bigg( -\int_t^T f^{\rm for}(t,u) du \bigg),
\end{align*}
where $f^{\rm dom}(t,T)$ denote the domestic forward rates, and where $f^{\rm for}(t,T)$ denote the foreign forward rates. Performing the Musiela type transformations $r_t^{\rm dom}(\xi) := f^{\rm dom}(t,t+\xi)$ and $r_t^{\rm for}(\xi) := f^{\rm for}(t,t+\xi)$, we obtain 
\begin{align}\label{FX-bond-1}
P(t,T) &= \exp \bigg( -\int_0^{T-t} r_t^{\rm dom}(u) du \bigg),
\\ \label{FX-bond-2} Q(t,T) &= \exp \bigg( -\int_0^{T-t} r_t^{\rm for}(u) du \bigg),
\end{align}
and we arrive at an SPDE of the type
\begin{align}\label{SPDE-FX}
\left\{
\begin{array}{rcl}
dR_t & = & \big( A R_t + \alpha(R_t) \big) dt + \sigma(R_t) dW_t + \int_E \gamma(R_{t-},x) (N(dt,dx) - F(dx)dt) \medskip
\\ R_0 & = & h_0.
\end{array}
\right.
\end{align}
for the process $R := (r^{\rm dom},r^{\rm for},X)$ with generator $A := (d/d\xi,d/d\xi,0)$. We choose the state space is $\calh := H \times H \times \bbr$, where $H := L_w^2(\bbr_+)$ with an admissible weight function $w : \bbr_+ \to (0,\infty)$. Let us consider the closed convex cone
\begin{align*}
\calk := \{ h \in \calh : 0 \leq h_1 \leq h_2 \text{ and } h_3 \geq 0 \}.
\end{align*}
If $\calk$ is invariant for the FX-like equation (\ref{SPDE-FX}), then the forward rates are nonnegative and condition (\ref{FX-cond-2}) is fulfilled. Furthermore, taking into account (\ref{FX-bond-1}) and (\ref{FX-bond-2}), we see that condition (\ref{FX-cond-1}) is also satisfied.

\begin{proposition}\label{prop-FX}
Suppose we have (\ref{main-1-prod}), that for all $(h^*,h) \in K \times \calk$ with $\la h^*,h_1 \ra_H = 0$ we have
\begin{align*}
&\la h^*,\alpha_1(h) \ra_H - \int_E \la h^*,\gamma_1(h,x) \ra_H F(dx) \geq 0,
\\ &\la h^*,\sigma_1^j(h) \ra_H = 0, \quad j \in \bbn,
\end{align*}
for all $(h^*,h) \in K \times \calk$ with $\la h^*,h_2-h_1 \ra_H = 0$ we have
\begin{align*}
&\la h^*,\alpha_2(h)-\alpha_1(h) \ra_H - \int_E \la h^*,\gamma_2(h,x)-\gamma_1(h,x) \ra_H F(dx) \geq 0,
\\ &\la h^*,\sigma_2^j(h) - \sigma_1^j(h) \ra_H = 0, \quad j \in \bbn,
\end{align*}
and for all $(h^*,h) \in \bbr_+ \times \calk$ with $h_3 = 0$ we have
\begin{align}\label{FX-cond-X-1}
&\alpha_3(h) - \int_E \gamma_3(h,x) F(dx) \geq 0,
\\ \label{FX-cond-X-2} &\sigma_3^j(h) = 0, \quad j \in \bbn.
\end{align}
Then the closed convex cone $\calk$ is invariant for the FX-like equation (\ref{SPDE-FX}).
\end{proposition}

\begin{proof}
Note that the cone $\calk$ has the product structure $\calk = \calk_1 \times \calk_2$, where
\begin{align*}
\calk_1 = \{ h \in H \times H : 0 \leq h_1 \leq h_2 \} \quad \text{and} \quad \calk_2 = \bbr_+.
\end{align*}
Let $R \in L(H \times H)$ be the linear isomorphism given by
\begin{align*}
Rh := ( h_1, h_2 - h_1 ), \quad h \in H \times H.
\end{align*}
According to Lemmas \ref{lemma-inv-matrix-1} and \ref{lemma-inv-matrix-2}, the cone $\calk_1$ has the generating system
\begin{align*}
\calg_1 = R^* ( K \times \{ 0 \} \cup \{ 0 \} \times K ),
\end{align*}
and, taking into account Proposition \ref{prop-seq-trans}, the cone $(\calk_1,\calg_1)$ is approximately generated by an unconditional Schauder basis. Furthermore, the cone $\calk_2$ has the generating system $\calg_2 = \bbr_+$. According to Proposition \ref{prop-prod-Schauder} the cone $\calk$ has the generating system
\begin{align*}
\calg &= \calg_1 \times \{ 0 \} \cup \{ 0 \} \times \calg_2
\\ &= \big( R^* ( K \times \{ 0 \} \cup \{ 0 \} \times K ) \times \{ 0 \} \big) \cup \big( \{ 0 \} \times \bbr_+ \big),
\end{align*}
and the cone $(\calk,\calg)$ is approximately generated by an unconditional Schauder basis. Consequently, taking into account Lemmas \ref{lemma-prod-inner-prod} and \ref{lemma-inv-matrix-3}, the result is a consequence of Theorem \ref{thm-general}.
\end{proof}

\begin{remark}
Suppose that the conditions from Proposition \ref{prop-FX} are fulfilled. As we have seen above, the recovery rate $S$ starts at one, which means that $X_0 = 0$. Therefore, we should start the FX-like equation (\ref{SPDE-FX}) with an initial condition $h \in \calk$ such that $h_3(0) = 0$. Moreover, as pointed out in \cite{Krabichler}, the recovery rate $S$ generally even starts with a constant trajectory at level one, which means that $X^{\tau} = 0$ for some strictly positive stopping time $\tau$. This can be assured by replacing conditions (\ref{FX-cond-X-1}) and (\ref{FX-cond-X-2}) by the stronger condition that for all $h \in \calk$ we have
\begin{align*}
&\int_E |\gamma_3(h,x)| F(dx) < \infty,
\\ &\alpha_3(h) - \int_E \gamma_3(h,x) F(dx) = 0,
\\ &\sigma_3^j(h) = 0, \quad j \in \bbn.
\end{align*}
\end{remark}

\begin{appendix}

\section{Generalized Filipovi\'{c} spaces}\label{app-Filipovic-space}

In this appendix we provide generalizations of the Filipovi\'{c} space from \cite{fillnm}. We are in particular interested in such function spaces of two variables. In order to prepare the required background, we start with a review of the one-dimensional situation.

Let $I \subset \bbr$ be an interval. We call a continuous function $w : I \to [1,\infty)$ an \emph{admissible weight function} if it is continuous, increasing and satisfies $w^{-1} \in L^1(I)$. We fix an arbitrary $x_0 \in I$ and define the \emph{Filipovi\'{c} space} $H_w(I)$ as the space of all absolutely continuous functions $h : I \to \bbr$ such that
\begin{align}\label{norm-Fil-1}
\| h \|_{w,x_0} := \bigg( |h(x_0)|^2 + \int_I |h'(x)|^2 w(x) dx \bigg)^{1/2} < \infty.
\end{align}
As we will see later on, the definition of the Filipovi\'{c} space $H_w(I)$ does not depend on the choice of $x_0 \in I$.

\begin{remark}
In \cite{fillnm} the situation $I = \bbr_+$ and $x_0 = 0$ was considered, and it was assumed that the weight function $w : I \to [1,\infty)$ is an increasing $C^1$-function such that $w^{-\frac{1}{3}} \in L^1(\bbr_+)$. For our purposes, it is enough if $w : I \to [1,\infty)$ is an admissible weight function as defined above.
\end{remark}

The proof of the following result is straightforward; cf. \cite[p. 76]{fillnm}.

\begin{proposition}
The linear mapping $T : \bbr \times L^2(I) \to H_w(I)$ given by
\begin{align*}
T(a,f) := a + \int_{x_0}^{\bullet} f(x) w^{-\frac{1}{2}}(x) dx
\end{align*}
is an isometric isomorphism with inverse
\begin{align*}
T^{-1}h = \big( h(x_0), h' w^{\frac{1}{2}} \big).
\end{align*}
\end{proposition}

Consequently, the Filipovi\'{c} space $H_w(I)$ is a separable Hilbert space. For the proof of the following result we can follow the arguments from \cite[p. 77]{fillnm}.

\begin{proposition}\label{prop-C1-C2-1}
The following statements are true:
\begin{enumerate}
\item For each $h \in H_w(I)$ we have the estimate
\begin{align*}
\| h' \|_{L^1(I)} \leq \| h \|_{w,x_0} \| w^{-1} \|_{L^1(I)}^{\frac{1}{2}}. 
\end{align*}

\item There is a constant $C > 0$ such that
\begin{align*}
\| h \|_{L^{\infty}(I)} \leq C \| h \|_{w,x_0}, \quad h \in H_w(I).
\end{align*}
\end{enumerate}
\end{proposition}

Now, we define $\overline{I} \subset [-\infty,\infty]$ as 
\begin{align}\label{I-closure}
\overline{I} := I \cup \{ \inf I \} \cup \{ \sup I \}. 
\end{align}
In the following result, the space $\bbr$ corresponds to all constant functions, and may thus be considered as a subspace of $H_w(I)$. The proof is an immediate consequence of Proposition \ref{prop-C1-C2-1}.

\begin{proposition}\label{prop-direct-sum-1}
For each $x \in \overline{I}$ the following statements are true:
\begin{enumerate}
\item The limit $h(x) := \lim_{\genfrac{}{}{0pt}{}{\xi \to x}{\xi \in I}}h(\xi)$ exists for each $h \in H_w(I)$.

\item There exists $\delta_{x} \in H_w(I)$ such that
\begin{align*}
h(x) = \la \delta_{x},h \ra_{w,x_0}, \quad h \in H_w(I).
\end{align*}
\item The norms $\| h \|_{w,x_0}$ and $\| h \|_{w,x}$ defined according to (\ref{norm-Fil-1}) are equivalent.

\item The subspaces $\bbr$ and
\begin{align}\label{H-w-x}
H_w^{x}(I) := \{ h \in H_w(I) : h(x) = 0 \}
\end{align}
are closed subspaces which are orthogonal with respect to the norm $\| \cdot \|_{w,x}$.

\item We have the direct sum decomposition $H_w(I) = \bbr \oplus H_w^{x}(I)$.
\end{enumerate}
\end{proposition}

Now, we assume that $\infty \in \overline{I}$. Then we can introduce the translation semigroup $(S_t)_{t \geq 0}$ given by
\begin{align*}
S_t h := h(t + \bullet), \quad t \geq 0
\end{align*}
for each $h \in H_w(I)$. Adopting the arguments from \cite[p. 78,79]{fillnm} and \cite[Lemma 3.5]{Benth-Kruehner}, we obtain the following result.

\begin{proposition}\label{prop-Fil-1-dim-semi}
$(S_t)_{t \geq 0}$ is a pseudo-contractive $C_0$-semigroup on $H_w(I)$.
\end{proposition}

In addition, the following result is straightforward to prove.

\begin{proposition}\label{prop-Fil-1-dim-semi-2}
Let $x \in \overline{I}$ and $H_0 \subset H_w(I)$ be a closed subspace which is invariant for the semigroup $(S_t)_{t \geq 0}$ such that $h(x) = 0$ for each $h \in H_0$. Then $(S_t)_{t \geq 0}$ is a $C_0$-semigroup of contractions on $H_0$ with respect to the norm $\| \cdot \|_{w,x}$.
\end{proposition}

Now, we will introduce Filipovi\'{c} spaces of two variables. For this purpose, let $I,J \subset \bbr$ be two intervals, and let $w : I \to [1,\infty)$ and $v : J \to [1,\infty)$ be two admissible weight functions. We fix an arbitrary $(x_0,y_0) \in I \times J$ and define the \emph{Filipovi\'{c} space} $H_{w,v}(I \times J)$ as the space of all functions $h : I \times J \to \bbr$ such that the following conditions are fulfilled:
\begin{itemize}
\item For each $y \in J$ the function $h(\cdot,y) : I \to \bbr$ is absolutely continuous. Consequently, for each $y \in J$ the function $h(\cdot,y) : I \to \bbr$ is almost everywhere differentiable, and hence there is a function $h_{\xi}(\cdot,y) : I \to \bbr$ such that
\begin{align*}
\partial_x h(x,y) = h_{\xi}(x,y) \quad \text{for almost every $x \in I$.}
\end{align*}
This gives us a function $h_{\xi} : I \times J \to \bbr$.

\item For each $x \in I$ the function $h_{\xi}(x,\cdot) : J \to \bbr$ is absolutely continuous. Consequently, for each $x \in I$ the function $h_{\xi}(x,\cdot) : J \to \bbr$ is almost everywhere differentiable, and hence there is a function $h_{\eta \xi}(x,\cdot) : J \to \bbr$ such that
\begin{align*}
\partial_{y} h_{\xi}(x,y) = h_{\eta \xi}(x,y) \quad \text{for almost every $y \in J$.}
\end{align*}
This gives us a function $h_{\eta \xi} : I \times J \to \bbr$.

\item For each $x \in I$ the function $h(x,\cdot) : J \to \bbr$ is absolutely continuous. Consequently, for each $x \in I$ the function $h(x,\cdot) : J \to \bbr$ is almost everywhere differentiable, and hence there is a function $h_{\eta}(x,\cdot) : J \to \bbr$ such that
\begin{align*}
\partial_y h(x,y) = h_{\eta}(x,y) \quad \text{for almost every $y \in J$.}
\end{align*}
This gives us a function $h_{\eta} : I \times J \to \bbr$.

\item For each $y \in J$ the function $h_{\eta}(\cdot,y) : I \to \bbr$ is absolutely continuous. Consequently, for each $y \in J$ the function $h_{\eta}(\cdot,y) : I \to \bbr$ is almost everywhere differentiable, and hence there is a function $h_{\xi \eta}(\cdot,y) : I \to \bbr$ such that
\begin{align*}
\partial_x h_{\eta}(x,y) = h_{\xi \eta}(x,y) \quad \text{for almost every $x \in I$.}
\end{align*}
This gives us a function $h_{\xi \eta} : I \times J \to \bbr$.

\item We have $h_{\xi \eta} = h_{\eta \xi}$.

\item We have
\begin{equation}\label{norm-Fil-2}
\begin{aligned}
\| h \|_{(w,v),(x_0,y_0)} &:= \bigg( |h(x_0,y_0)|^2 + \int_I |h_{\xi}(x,y_0)|^2 w(x) dx + \int_J |h_{\eta}(x_0,y)|^2 v(y) dy
\\ &\qquad + \int_{I \times J} |h_{\xi \eta}(x,y)|^2 w(x) v(y) dx dy \bigg)^{1/2} < \infty.
\end{aligned}
\end{equation}
\end{itemize}

\begin{remark}
For each $h \in C^2(I \times J)$ with $\| h \|_{(w,v),(x_0,y_0)} < \infty$ we have $h \in H_{w,v}(I \times J)$ with $h_{\xi} = \partial_x h$, $h_{\eta} = \partial_y h$ and $h_{\xi \eta} = \partial_{xy} h$.
\end{remark}

\begin{proposition}
The linear mapping
\begin{align*}
T : \bbr \times L^2(I) \times L^2(J) \times L^2(I \times J) \to H_{w,v}(I \times J)
\end{align*}
given by
\begin{align*}
T(a,f_1,f_2,g)(x,y) &:= a + \int_{x_0}^{x} f_1(\xi) w^{-\frac{1}{2}}(\xi) d\xi + \int_{y_0}^{y} f_2(\eta) v^{-\frac{1}{2}}(\eta) d\eta
\\ &\qquad + \int_{x_0}^{x} \int_{y_0}^{y} g(\xi,\eta) w^{-\frac{1}{2}}(\xi) v^{-\frac{1}{2}}(\eta) d\eta d\xi
\end{align*}
is an isometric isomorphism with inverse
\begin{align}\label{inverse-Fil-2}
T^{-1} h = \big( h(x_0,y_0), h_{\xi}(\cdot,y_0) w^{\frac{1}{2}}, h_{\eta}(x_0,\cdot) v^{\frac{1}{2}}, h_{\xi \eta} w^{\frac{1}{2}} v^{\frac{1}{2}} \big).
\end{align}
\end{proposition}

\begin{proof}
Let $(a,f_1,f_2,g) \in \bbr \times L^2(I) \times L^2(J) \times L^2(I \times J)$ be arbitrary and define the function $h := T(a,f_1,f_2,g)$. Then the conditions above are satisfied with $h_{\xi}$, $h_{\eta}$ and $h_{\xi \eta}$ given by
\begin{align*}
h_{\xi}(x,y) &= f_1(x) w^{-\frac{1}{2}}(x) + w^{-\frac{1}{2}}(x) \int_{y_0}^{y} g(x,\eta) v^{-\frac{1}{2}}(\eta) d\eta,
\\ h_{\eta}(x,y) &= f_2(y) v^{-\frac{1}{2}}(y) + v^{-\frac{1}{2}}(y) \int_{x_0}^x g(\xi,y) w^{-\frac{1}{2}}(\xi) d\xi,
\\ h_{\xi \eta}(x,y) &= g(x,y) w^{-\frac{1}{2}}(x) v^{-\frac{1}{2}}(y).
\end{align*}
In particular, we have
\begin{align*}
h_{\xi}(x,y_0) &= f_1(x) w^{-\frac{1}{2}}(x),
\\  h_{\eta}(x_0,y) &= f_2(y) v^{-\frac{1}{2}}(y).
\end{align*}
Hence, we have $h \in H_{w,v}(I \times J)$, and $T$ is an isometry. Moreover, it is straightforward to check that its inverse $T^{-1}$ is given by (\ref{inverse-Fil-2}).
\end{proof}

Consequently, the Filipovi\'{c} space $H_{w,v}(I \times J)$ is a separable Hilbert space.

\begin{proposition}\label{prop-Fil-C1-C2}
The following statements are true:
\begin{enumerate}
\item For each $h \in H_{w,v}(I \times J)$ we have the estimates
\begin{align*}
\| h_{\xi}(\cdot,y_0) \|_{L^1(I)} &\leq \| h \|_{(w,v),(x_0,y_0)} \| w^{-1} \|_{L^1(I)}^{\frac{1}{2}},
\\ \| h_{\eta}(x_0,\cdot) \|_{L^1(J)} &\leq \| h \|_{(w,v),(x_0,y_0)} \| v^{-1} \|_{L^1(J)}^{\frac{1}{2}},
\\ \| h_{\xi \eta} \|_{L^1(I \times J)} &\leq \| h \|_{(w,v),(x_0,y_0)} \| w^{-1} \|_{L^1(I)}^{\frac{1}{2}} \| v^{-1} \|_{L^1(J)}^{\frac{1}{2}}.
\end{align*}
\item There is a constant $C > 0$ such that
\begin{align*}
\| h \|_{L^{\infty}(I \times J)} \leq C \| h \|_{(w,v),(x_0,y_0)}, \quad h \in H_{w,v}(I \times J).
\end{align*}
\end{enumerate}
\end{proposition}

\begin{proof}
The first statement is a consequence of the Cauchy Schwarz inequality (cf. the arguments from \cite[p. 77]{fillnm}). For the proof of the second statement, let $(x,y) \in I \times J$ be arbitrary. Then we have
\begin{align*}
h(x,y) - h(x_0,y_0) &= h(x,y) - h(x_0,y) + h(x_0,y) - h(x_0,y_0)
\\ &= \int_{x_0}^{x} h_{\xi}(\xi,y) d\xi + \int_{y_0}^{y} h_{\eta}(x_0,\eta) d\eta
\\ &= \int_{x_0}^{x} \big( h_{\xi}(\xi,y) - h_{\xi}(\xi,y_0) + h_{\xi}(\xi,y_0) \big) d\xi + \int_{y_0}^{y} h_{\eta}(x_0,\eta) d\eta
\\ &= \int_{x_0}^{x} \int_{y_0}^{y} h_{\xi \eta}(\xi,\eta) d\eta d\xi + \int_{x_0}^{x} h_{\xi}(\xi,y_0) d\xi + \int_{y_0}^{y} h_{\eta}(x_0,\eta) d\eta.
\end{align*}
By virtue of the first statement, this completes the proof.
\end{proof}

Now, we define $\overline{I}, \overline{J} \subset [-\infty,\infty]$ according to (\ref{I-closure}). In the following result, the spaces $H_w^x(I)$ and $H_v^y(J)$ are the subspaces of the one-dimensional Filipovi\'{c} spaces defined according to (\ref{H-w-x}), and may thus be considered as subspaces of the two-dimensional Filipovi\'{c} space $H_{w,v}(I \times J)$. The proof is an immediate consequence of Proposition \ref{prop-Fil-C1-C2}.

\begin{proposition}\label{prop-Fil-2-dim-decomp}
For each $(x,y) \in \overline{I} \times \overline{J}$ the following statements are true:
\begin{enumerate}
\item The limit $h(x,y) := \lim_{\genfrac{}{}{0pt}{}{(\xi,\eta) \to (x,y)}{(\xi,\eta) \in I \times J}}h(x,y)$ exists for each $h \in H_{w,v}(I \times J)$.

\item There exists $\delta_{(x,y)} \in H_{w,v}(I \times J)$ such that
\begin{align*}
h(x,y) = \la \delta_{(x,y)},h \ra_{(w,v),(x,y)}, \quad h \in H_{w,v}(I \times J).
\end{align*}
\item The norms $\| h \|_{(w,v),(x_0,y_0)}$ and $\| h \|_{(w,v),(x,y)}$ defined according to (\ref{norm-Fil-2}) are equivalent.

\item The subspaces $\bbr$, $H_w^x(I)$, $H_v^y(J)$ and
\begin{align*}
H_{w,v}^{(x,y)}(I \times J) := \{ h \in H_{w,v} : h(x,y) = 0, h_{\xi}(\cdot,y) = 0, h_{\eta}(x,\cdot) = 0 \}
\end{align*}
are closed subspaces which are mutual orthogonal with respect to the norm $\| \cdot \|_{(w,v),(x,y)}$.

\item We have the direct sum decomposition
\begin{align*}
H_{w,v}(I \times J) = \bbr \oplus H_w^x(I) \oplus H_v^y(J) \oplus H_{w,v}^{(x,y)}(I \times J).
\end{align*}
\end{enumerate}
\end{proposition}

Now, we assume that $\infty \in \overline{I}$. Then we can introduce the translation semigroup $(S_t)_{t \geq 0}$ given by
\begin{align*}
S_t h(x,y) := h(x+t,y), \quad t \geq 0
\end{align*}
for each $h \in H_{w,v}(I \times J)$. With similar techniques as in the one-dimensional case, we establish the following two results.

\begin{proposition}\label{prop-Fil-2-dim-semi}
$(S_t)_{t \geq 0}$ is a pseudo-contractive $C_0$-semigroup on $H_{w,v}(I \times J)$.
\end{proposition}

\begin{proposition}\label{prop-Fil-2-dim-semi-2}
Let $(x,y) \in \overline{I} \times \overline{J}$ and $H_0 \subset H_{w,v}(I \times J)$ be a closed subspace which is invariant for the semigroup $(S_t)_{t \geq 0}$ such that
$h(x,y) = 0$ for each $h \in H_0$. Then $(S_t)_{t \geq 0}$ is a $C_0$-semigroup of contractions on $H_0$ with respect to the norm $\| \cdot \|_{(w,v),(x,y)}$.
\end{proposition}

\section{Orthogonal projections in $L^2$-spaces}\label{app-proj-L2}

In this appendix we provide the required results about orthogonal projections in $L^2$-spaces. Let $(X,\calx,\mu)$ be a measure space, and define the Hilbert space $H := L^2(X,\calx,\mu)$.

\begin{lemma}\label{lemma-MT-prob-space}
If $(X,\calx,\mu)$ is a probability space, then for each sub-$\sigma$-algebra $\calg \subset \calx$ the orthogonal projection on $L^2(\calg)$ is given by $\pi_{L^2(\calg)} = \bbe[ \cdot | \calg ]$.
\end{lemma}

\begin{proof}
This follows from \cite[Cor. 8.17]{Klenke}.
\end{proof}

\begin{proposition}\label{prop-MT-prob-space}
Suppose that $(X,\calx,\mu)$ is a probability space. Let $(\calg_n)_{n \in \bbn}$ be a filtration such that $\calx = \sigma(\bigcup_{n \in \bbn} \calg_n)$. Then for each $h \in H$ we have
\begin{align*}
\pi_{L^2(\calg_n)} h \to h \quad \text{in $H$.}
\end{align*}
\end{proposition}

\begin{proof}
This is a consequence of the martingale convergence theorem (see \cite[Thm. 11.10]{Klenke}) and Lemma \ref{lemma-MT-prob-space}.
\end{proof}

\begin{lemma}\label{lemma-proj-fin-measure-space}
Suppose that $(X,\calx,\mu)$ is a finite measure space with $\mu \neq 0$. Let $\bbp$ be the probability measure on $(X,\calx)$ given by
\begin{align*}
\bbp(B) = \frac{\mu(B)}{\mu(X)}, \quad B \in \calx.
\end{align*}
Then the following statements are true:
\begin{enumerate}
\item We have $L^2(\mu) = L^2(\bbp)$ as sets, and 
\begin{align*}
\| h \|_{L^2(\mu)} = \sqrt{\mu(X)} \| h \|_{L^2(\bbp)}, \quad h \in H.
\end{align*}
\item In particular, $\| \cdot \|_{L^2(\mu)}$ and $\| \cdot \|_{L^2(\bbp)}$ are equivalent norm on $H$.

\item Let $\calg \subset \calx$ be a sub-$\sigma$-algebra. We set $L^2(\calg) := L^2(X,\calg,\mu) = L^2(X,\calg,\bbp)$. Let $\pi_{L^2(\calg)}^{\mu}$ be the orthogonal projection on $L^2(\calg)$ with respect to $\| \cdot \|_{L^2(\mu)}$, and let $\pi_{L^2(\calg)}^{\bbp}$ be the orthogonal projection on $L^2(\calg)$ with respect to $\| \cdot \|_{L^2(\bbp)}$. Then we have $\pi_{L^2(\calg)}^{\mu} = \pi_{L^2(\calg)}^{\bbp}$.
\end{enumerate}
\end{lemma}

\begin{proof}
The first two statements are evident. For each $h \in L^2(\mu)$ we have
\begin{align*}
\inf_{g \in L^2(\calg)} \| g - h \|_{L^2(\mu)} &= \sqrt{\mu(X)} \inf_{g \in L^2(\calg)} \| g - h \|_{L^2(\bbp)}
\\ &= \sqrt{\mu(X)} \| \pi_{L^2(\calg)}^{\bbp} h - h \|_{L^2(\bbp)}
\\ &= \| \pi_{L^2(\calg)}^{\bbp} h - h \|_{L^2(\mu)},
\end{align*}
showing that $\pi_{L^2(\calg)}^{\mu} = \pi_{L^2(\calg)}^{\bbp}$.
\end{proof}

\begin{proposition}\label{prop-MT-finite-space}
Suppose that $(X,\calx,\mu)$ is a finite measure space. Let $(\calg_n)_{n \in \bbn}$ be a filtration such that $\calx = \sigma(\bigcup_{n \in \bbn} \calg_n)$. Then for each $h \in H$ we have
\begin{align*}
\pi_{L^2(\calg_n)} h \to h \quad \text{in $H$.}
\end{align*}
\end{proposition}

\begin{proof}
This is a consequence of Proposition \ref{prop-MT-prob-space} and Lemma \ref{lemma-proj-fin-measure-space}.
\end{proof}

The proof of the following auxiliary result is straightforward, and therefore omitted.

\begin{lemma}\label{lemma-isom-Leb}
Let $B \in \calx$ be arbitrary, and define the Hilbert space $H_B := L^2(B) := L^2(B,\calx \cap B,\mu)$. We define the mapping $T_B : H_B \to H$ as $T_B h = h \bbI_B$. More precisely, for any representative $h \in \call^2(B)$ a repräsentative $g \in \call^2(X)$ of $T_B h$ is given by
\begin{align*}
g(x) = 
\begin{cases}
h(x), & x \in B,
\\ 0, & x \in B^c.
\end{cases}
\end{align*}
Then the following statements are true:
\begin{enumerate}
\item The mapping $T_B : H_B \to H$ is a linear isometry.

\item We have $\ran(T_B) = \{ h \in H : h \bbI_{B^c} = 0 \}$.

\item For each $h \in \ran(T_B)$ we have $T_B^{-1} h = h|_B$.

\item For each $h \in H_B$ we have $h \geq 0$ if and only if $T_B h \geq 0$.

\item Let $\calg \subset \calx$ be a sub-$\sigma$-algebra such that $B \in \calg$. Then we have
\begin{align*}
T_B \big( L^2(B,\calg \cap B) \big) = \{ h \in L^2(X,\calg) : h \bbI_{B^c} = 0 \}.
\end{align*}
\end{enumerate}
\end{lemma}

\begin{lemma}\label{lemma-proj-indicator}
Let $\calg$ be a sub-$\sigma$-algebra, and let $B \in \calg$ be arbitrary. Then for each $h \in H$ we have
\begin{align}\label{pi-on-set-B}
\pi_{L^2(X,\calg)}(h \bbI_B) = T_B \, \pi_{L^2(B,\calg \cap B)}(h|_B).
\end{align}
\end{lemma}

\begin{proof}
Let $h \in H$ be arbitrary. Using Lemma \ref{lemma-isom-Leb}, for all $g \in L^2(B,\calg \cap B)$ we have
\begin{align*}
\| g - h|_B \|_{L^2(B)} = \| T_B g - T_B(h|_B) \|_{L^2(X)} = \| T_B g - h \bbI_B \|_{L^2(X)},
\end{align*}
showing that
\begin{align*}
\inf_{g \in  L^2(B,\calg \cap B)} \| g \bbI_B - h \bbI_B \|_{L^2(X)} = \inf_{g \in  L^2(B,\calg \cap B)} \| g -  h|_B \|_{L^2(B)}.
\end{align*}
Furthermore, for all $f \in L^2(X,\calg)$ we have
\begin{align*}
\| f - h \bbI_B \|_{L^2(X)} \geq \| f \bbI_B - h \bbI_B \|_{L^2(X)},
\end{align*}
which proves (\ref{pi-on-set-B}).
\end{proof}

The following result applies to certain $\sigma$-finite measure spaces.

\begin{proposition}\label{prop-conv-proj-sigma-fin}
Let $(\calg_n)_{n \in \bbn}$ be a filtration such that $\calx = \sigma(\bigcup_{n \in \bbn} \calg_n)$. We assume there exists a sequence $(B_m)_{m \in \bbn} \subset \bigcup_{n \in \bbn} \calg_n$ such that $B_m \uparrow X$ and $\mu(B_m) < \infty$ for all $m \in \bbn$. Then for each $h \in H$ we have
\begin{align*}
\pi_{L^2(\calg_n)} h \to h \quad \text{in $H$.}
\end{align*}
\end{proposition}

\begin{proof}
Let $h \in H$ be arbitrary, and let $\epsilon > 0$ be arbitrary. Since $B_m \uparrow X$, by Lebesgue's dominated convergence theorem there exists $m \in \bbn$ such that
\begin{align}\label{eps-3-1}
\| h - h \bbI_{B_m} \|_H \leq \frac{\epsilon}{3}.
\end{align}
By Proposition \ref{prop-MT-finite-space} there exists $n \in \bbn$ with $B_m \in \calg_n$ such that
\begin{align*}
\| h|_{B_m} - \pi_{L^2(B_m,\calg_n \cap B_m)}(h|_{B_m}) \|_{L^2(B_m)} \leq \frac{\epsilon}{3}.
\end{align*}
By Lemmas \ref{lemma-isom-Leb} and \ref{lemma-proj-indicator} we obtain
\begin{align}\label{eps-3-2}
\| h \bbI_{B_m} - \pi_{L^2(X,\calg_n)}(h \bbI_{B_m}) \|_{H} \leq \frac{\epsilon}{3}.
\end{align}
Furthermore, by (\ref{eps-3-1}) we have
\begin{align}\label{eps-3-3}
\| \pi_{L^2(X,\calg_n)}(h \bbI_{B_m}) - \pi_{L^2(X,\calg_n)}(h) \|_H \leq \frac{\epsilon}{3}.
\end{align}
Combining (\ref{eps-3-1})--(\ref{eps-3-3}) we arrive at
\begin{align*}
\| h - \pi_{L^2(\calg_n)}(h) \|_H \leq \epsilon,
\end{align*}
completing the proof.
\end{proof}

\end{appendix}

\end{document}